\documentclass[reqno, 11pt]{amsart}%
\usepackage{amsaddr} 
\usepackage{amssymb}
\usepackage[nesting]{hyperref}
\usepackage[pdftex]{graphicx}
\usepackage{listings}
\usepackage{multirow}
\usepackage{placeins}
\usepackage{color}
\usepackage{lscape}
\usepackage{dsfont}
\usepackage{amsmath}
\usepackage{amsfonts}
\usepackage{tikz}
\usepackage{verbatim}
\usepackage{accents} 
\usepackage{todonotes}
\usepackage{mathtools}
\usepackage{bbm}
\usepackage{natbib}
\usepackage{lipsum}
\usepackage{caption}
\usepackage{subcaption}
\usepackage{afterpage}

\allowdisplaybreaks

\DeclareCaptionLabelFormat{subfigurelabel}{\textnormal{(\alph{subfigure})}} 
\newcommand{\BlackBoxes}{\global\overfullrule5pt}

\BlackBoxes

\newcommand{\R}{\mathbb{R}} 
\newcommand{\N}{\mathbb{N}} 

\newcommand{\PP}{\mathbb{P}}
\newcommand{\EE}{\mathbb{E}}
\newcommand{\E}{\mathbb{E}}

\newcommand{\1}{\mathbbm{1}}

\mathtoolsset{showonlyrefs}

\newtheorem{theorem}{Theorem}

\newtheorem{lemma}[theorem]{Lemma}

\theoremstyle{definition}

\newtheorem{remark}[theorem]{Remark}
\newtheorem{definition}[theorem]{Definition}

\numberwithin{equation}{section} \numberwithin{theorem}{section}
\def\0{\kern0pt\-\nobreak\hskip0pt\relax}

\makeatletter
\AtBeginDocument{ \def\@serieslogo{ \vbox to\headheight{ \parindent\z@ \fontsize{6}{7\p@}\selectfont
\vss}}}

\def\makeoverbar#1#2#3#4#5#6#7{ \setbox0=\hbox{$\m@th#2\mkern#5mu{{}#3{}}\mkern#6mu$} \setbox1=\null \dimen@=#4\fontdimen8#13 \dimen@=3.5\dimen@
\advance\dimen@ by \ht0 \dimen@=-#7\dimen@ \advance\dimen@ by \wd0
\ht1=\ht0 \dp1=\dp0 \wd1=\dimen@
\dimen@=\fontdimen8#13 \fontdimen8#13=#4\fontdimen8#13
\rlap{\hbox to \wd0{$\m@th\hss#2{\overline{\box1}}\mkern#5mu$}}
\fontdimen8#13=\dimen@}
\def\mylabel#1#2{{\def\@currentlabel{#2}\label{#1}}}
\makeatother
\begin{document}
\title[Relative portfolio optimization via a VaR-based constraint]{Relative portfolio optimization via a value at risk based constraint}
\author[N. \smash{B\"auerle}]{Nicole B\"auerle$^*$}
\address[N. B\"auerle]{Department of Mathematics,
Karlsruhe Institute of Technology (KIT), D-76128 Karlsruhe, Germany, \href{mailto:nicole.baeuerle@kit.edu}{nicole.baeuerle@kit.edu}, ORCID 0000-0003-0077-3444\footnote{$^*$Corresponding author}}


\author[T. \smash{G\"oll}]{Tamara G\"oll}
\address[T. G\"oll]{Department of Mathematics,
Karlsruhe Institute of Technology (KIT), D-76128 Karls\-ruhe, Germany, \href{mailto:tamara.goell@kit.edu} {tamara.goell@kit.edu}, ORCID 0009-0007-8341-723X}



%

\begin{abstract}
In this paper, we consider $n$ agents who invest in a general financial market that is free of arbitrage and complete. The aim of each investor is to maximize her expected utility  while ensuring, with a specified probability, that her terminal wealth exceeds a benchmark defined by her competitors' performance. This setup introduces an interdependence between agents, leading to a search for Nash equilibria. In the case of two agents and CRRA utility, we are able to derive all Nash equilibria in terms of terminal wealth. For $n>2$ agents and logarithmic utility we distinguish two cases. In the first case, the probabilities in the constraint are small and we can characterize all Nash equilibria. In the second case, the probabilities are larger and we look for Nash equilibria in a certain set.   We also discuss the impact of the competition using some numerical examples. As a by-product, we solve some portfolio optimization problems with probability constraints.
\end{abstract}
\maketitle


\makeatletter \providecommand\@dotsep{5} \makeatother



\vspace{0.5cm}
\begin{minipage}{14cm}
{\small
\begin{description}
\item[\rm \textsc{ Key words}]
{\small Nash equilibrium; competitive investment; Value at risk; CRRA utility}

\end{description}
}
\end{minipage}\\

\textit{Statements and Declarations:} The authors have no relevant financial or non-financial interests to disclose. Data sharing is not applicable to this article as no datasets were generated or analyzed during the current study.

\section{Introduction}
In standard asset allocation problems, a single agent typically invests in a financial market to optimize an objective, such as expected utility or mean-variance. In reality, however, investors are rarely isolated; they are often influenced by their relative performance compared to others. In this paper, we examine a scenario with  $n$ agents investing in a shared financial market, where each agent aims to maximize her expected utility while ensuring, with a specified probability, that her terminal wealth exceeds a benchmark defined by her competitors' performance. For instance, an agent might require that, with $90\%$ probability, her final wealth surpasses the average wealth of her competitors. This setup introduces a strategic layer to the optimization problem, where agents' objectives are interdependent, leading to a search for Nash equilibria under probabilistic constraints. Consequently, our framework integrates strategic interactions among agents with value-at-risk constraints.

Before we have a closer look at our results, let us mention the related literature. There have been a number of papers recently which consider strategic interactions between agents in asset allocation problems. The motivation stems from \cite{brown2001careers} and \cite{kempf2008tournaments}. The majority of papers model the interaction by maximizing the expected utility of relative wealth in financial markets of Black Scholes type. Relative wealth may be measured by a linear function of the terminal wealths of all agents, see e.g.\ \cite{espinosa2015optimal} and \cite{espinosa2010stochastic}, or by a multiplicative function of the terminal wealths, see e.g.\ \cite{basak2014strategic}. {\color{black} More precisely, this means that $X_T^{i,\pi^i} - \frac{\theta_i}{n-1}\sum_{j\neq i}X_T^{j,\pi^j}$ or $X_T^{i,\pi^i} \big( \sum_{j\neq i}X_T^{j,\pi^j}\big)^{-\frac{\theta_i}{n-1}}$ enter the utility function of agent $i$ where $X_T^{i,\pi^i}$ is the  wealth of agent $i$ at the terminal time $T$. For example, \cite{basak2015competition} consider two agents investing in a Black Scholes model and maximizing the power utility of the ratio of both wealths. The difference between own wealth and arithmetic mean of other agents' wealth has been treated in \cite{espinosa2015optimal}. } 
\cite{lacker2019mean} investigate a model where each agent has her own financial market and consider the corresponding mean-field approach. This work has later been extended to investment-consumption in \cite{lacker2020many}. \cite{fu2023mean} and \cite{fu2023mean2} provide a relation between Nash equilibria and systems of forward backward stochastic differential equations for agents with power utility and multiplicative relative performance for investment and consumption. Forward utilities  with competition are for example considered in \cite{musiela2006investments}, \cite{dos2019forward} and \cite{anthropelos2022competition}. Models with more general financial market have been treated in \cite{bauerle2023nash}, \cite{kraft2020dynamic}, \cite{hu2021n}, \cite{aydougan2024optimal}. Papers which model in addition price impacts of the investors are \cite{bauerle2024nash} and \cite{curatola2024asset}.

On the other hand, there is a stream of literature considering benchmark and value-at-risk constraints for one agent problems. 
For example \cite{basak2001value} maximize expected utility under the constraint that, with a certain probability, the terminal wealth is above a given deterministic threshold. This is equivalent to requiring that  the value-at-risk at the terminal time of the wealth is below a threshold. \cite{gabih2005dynamic,gabih2009utility,gabih2006optimal} consider a similar  point of view  replacing the value-at-risk by expected shortfall and \cite{sass2010optimal} and \cite{bauerle2022optimal} extend this situation to problems with partial information. {\color{black} But in these models there is only a single agent and no competition involved.}

In our model, we combine these two aspects. We consider $n$ agents who invest into the same financial market. We do not need a specific model for this market. It should be free of arbitrage and complete. The agents aim at maximizing their expected utility at terminal time $T$. Most of the paper is on logarithmic utility, but we also consider power utility. As a constraint, the probability that one's wealth exceeds a linear combination of the other agents' wealth is bounded from below.
Thus, we adapt the model of \cite{basak2001value} to include relative concerns by replacing the constant solvency level in the optimization problem of agent $i$ by a weighted arithmetic mean of the other $n-1$ agents' terminal wealth.  {\color{black} The motivation for such a model is threefold: First there is a psychological argument for private investors. Whether or not a return feels like a gain or loss is often relative to how other investors (friends, colleagues) performed or how the general stock market performed. This has for example been implemented in the cumulative prospect theory of \cite{tversky1992advances} by introducing a reference point. Thus, investors try to beat such reference points. Second, fund managers often receive a part of their bonuses for outperforming other funds or benchmarks, see e.g.\ \cite{browne1999reaching}. Thus, part of their attention is dedicated to outperform these benchmarks with at least a certain probability. Third, this can equivalently be considered as a portfolio optimization problem with a stochastic value-at-risk  constraint, i.e. a value-at-risk constraint where the risk level is stochastic. }

 \cite{bell1980competitive} consider a simple static two-person zero sum game where the payoff is the probability of beating the opponent's outcome. There are also links to what is known as Colonel Blotto games, see \cite{kovenock2021generalizations}. We incorporate the constraint 
\begin{equation}\label{eq:var_constraint}
\PP\Big(X_i \geq \sum_{j\neq i} \beta_{{\color{black}ij}} X_j\Big)\geq \alpha_i
\end{equation}
into the optimization problem of agent $i$, where $\alpha_i \in [0,1]$ and {\color{black}$\beta_{ij} \in [0,1]$ for all $i,j\in \{1,\ldots,n\}$}. By $X_j$ we denote the terminal wealth of agent $j\in \{1,\ldots,n\}$. Thus, in a fraction of $\alpha_i$ of the possible market scenarios, agent $i$ attains a terminal wealth which is at least as large as the weighted average of her competitors' terminal wealth. The objective function of agent $i$ is given by  the expected utility $\EE U(X_i)$ of her terminal wealth where $U$ is a CRRA utility.  
A possible choice is $\beta_{{\color{black}ij}} = \frac{1}{n-1}$, but we allow for more generality at this point. It is, for example, possible to consider weights in terms of the initial capital invested by the agents so that a larger initial investment goes along with a larger weight assigned to the corresponding agent. 
If $\alpha_i$ is chosen close to $1$, agent $i$ wants to insure her terminal wealth against the other agents' wealth in almost all possible scenarios, while a value $\alpha_i$ close to $0$ implies that she does not care as much about her performance with respect to the others. 
We look for Nash equilibria for this problem in terms of the final wealth. Depending on the specific parameters there are multiple or unique Nash equilibria. In the case of two agents, we obtain rather explicit results and are able to determine all Nash equilibria. In particular some discontinuity phenomena show up. In case of a logarithmic utility, as soon as some probabilities for the constraints are less than one, the structure of the classical optimal terminal wealth is enforced in the Nash equilibrium. The situation with many agents is considerably more difficult. We have to distinguish how large the probabilities for the constraints are. If the sum of all probabilities is at most one, the wealth constraints for the agents will be satisfied on disjoint events for all Nash equilibria.

This paper is organized as follows. In the next section, we comment on the underlying financial market. Section \ref{sec:two} deals with the competition of two agents. We consider logarithmic and power utility there and derive all Nash equilibria rather explicitly.  Section \ref{sec:multi} is then dedicated to the $n$ agent case with $n\ge 3.$ In this section, we have to distinguish the cases $\alpha_1+\ldots +\alpha_n\le 1$ and $\alpha_1+\ldots +\alpha_n> 1.$ While we are able to determine the structure of all Nash equilibria in the first case, we restrict to searching for Nash equilibria of specific type in the second case. 
The last section discusses some numerical results and the appendix provides some auxiliary results.

\section{Financial market}\label{sec:FM}
 We do not introduce a specific financial market, but assume that it is free of arbitrage and complete and that the underlying probability space is continuous. An important example is the Black Scholes market which consists of a riskless bond with interest rate $r$, where we set $r= 0$, and $d$ stocks. Thus, if $W=(W_1,\ldots,W_d)$ is a $d$-dimensional Brownian motion, then the price processes for stocks $k=1,\ldots,d$  are given by
\begin{align*}
\mathrm{d}S_k(t)=&S_k(t)\Big(\mu_k\,
\mathrm{d}t+\sum_{\ell =1}^d\sigma_{k\ell }\,\mathrm{d}W_{\ell }(t)\Big),
\end{align*}
where $S_k(0)=1$, $\sigma_{k\ell}\ge 0$. By $\mu = (\mu_1,\ldots,\mu_d)\in \R^d$ we denote the drift vector and by $\sigma = (\sigma_{k\ell})_{1\leq k,\ell \leq d}$ the volatility matrix, which has to be regular. But as mentioned before, this is just a special example. We assume that all random variables are defined on a probability space $(\Omega,\mathcal{F},\PP)$ where $\Omega$ is continuous, e.g. $\Omega=\R$ and $\mathcal{F}=\mathcal{F}_T^W$ is generated by the Brownian motions up to time $T$.
 By $Z_T$ we denote the state price density of the financial market. The price at time $t=0$ of a contingent claim $H$ which is an integrable, $\mathcal{F}$-measurable random variable, is thus given by $\EE[Z_TH].$  In the example of a Black Scholes market with zero interest,  this would be
$$ Z_T = \exp\Big( -\theta^\top W_T-\frac12 \| \theta\|^2 T \Big)$$
with $\theta =\sigma^{-1}\mu \in \R^d.$ Trading strategies are defined as $(\mathcal{F}_t^W)-$ progressively measurable processes $(\pi_t^k)$ denoting  the amount of money invested at time $t$ in stock $k$.
Since we compute the Nash equilibrium in terms of terminal wealths of the agents, we do not have to be specific about the financial market. The  $\mathcal{F}$-measurable random variable $X_i$ always denotes the terminal wealth of agent $i$ at time $T$. Since the financial market is complete, all  $\mathcal{F}$-measurable and integrable  random variables can be attained by a certain portfolio strategy. The initial wealth needed for this portfolio is given by its price $\EE[Z_T X_i].$ For more general semimartingale financial markets which fit into our setup see \cite{bauerle2023nash}.

{\color{black} 
\begin{remark}
In order to motivate the discussion in the next sections, let us consider the following situation within the Black Scholes market with one stock, i.e. $d=1.$  We have only one agent who tries to invest in such a way that she maximizes her expected utility  while beating (a fraction of) the stock price $S_T$ with at least a certain probability, i.e.\ our agent compares her wealth to the stock performance. More precisely, we consider the problem
$$(P_B) \left\{ \begin{array}{l}
\EE \ln (X) \to \max\\
\PP(X \ge  \beta S_T)\ge \alpha,\\
\EE[Z_TX]=x_0,\\
X \mbox{ is } \mathcal{F}-\mbox{measurable},\\
\end{array}\right.$$
where $X$ is the terminal wealth of the agent.
Thus, the agent maximizes the logarithmic utility of terminal wealth under the constraint that she beats at least with probability $\alpha$ the benchmark $\beta S_T, \beta \in [0,1]$. The initial wealth is $x_0.$ Here we only state  the optimal solution which can be derived with Lemma~\ref{lem:bound2}. We explain the derivation of results like this later, when we treat multi-agent models.
Define
\begin{align}
    \kappa(\lambda) := \Big(\frac{\lambda}{\beta}\Big)^{\frac{\mu}{\mu-\sigma^2}}\exp\Big(-\frac{1}{2}\mu T \Big), \quad \lambda \ge 0, 
\end{align}
and set  $\lambda_\alpha$   such that $ \PP(\beta S_T \le \lambda_\alpha/Z_T)=\alpha.$
Then the optimal terminal wealth is $X^* =x_0/Z_T$ if $x_0\ge \lambda_\alpha$. Thus, if the agent is rich enough, the strategy which simply maximizes the expected utility will also satisfy the constraint. In the case $x_0< \lambda_\alpha$, we obtain
$$ X^* =  \left\{ \begin{array}{cl}
\lambda/Z_T, & \mbox{ if } Z_T \le \kappa(\lambda_\alpha) \mbox{ or } Z_T > \kappa(\lambda),\\
\beta S_T,  & \mbox{ if } Z_T \in ( \kappa(\lambda_\alpha),  \kappa(\lambda)  ],
\end{array}\right.$$
where $\lambda > \lambda_\alpha$ is such that $\EE[X^*Z_T]=x_0.$ Thus, the terminal wealth is set equal to the reference $\beta S_T$ on a certain set of probability $\alpha$ and elsewhere the optimal wealth has the structure of the optimal wealth without constraint.  This problem is related to a number of similar questions which have been treated in the literature before such as maximizing the probability of beating a benchmark or minimizing shortfall \citep[see e.g.][]{browne1999reaching,korn2014portfolio, follmer2016stochastic}.
\end{remark}}

\section{Two agent case}\label{sec:two}
In this section, we consider the problem with two agents and two different utility functions: logarithmic and power utility. It turns out that there is a significant difference in the solution between these utility functions, though it is well-known that the logarithmic utility can be seen as a limiting case of  power utilities. 

\subsection{Logarithmic utility}
We first consider the case where we have $n=2$ agents with  a logarithmic utility. Both agents try to maximize their expected utility of terminal wealth $X_i, i=1,2,$ at time $T$, given a fixed initial capital $x_0^i>0, i=1,2,$ under the constraint that their respective terminal wealth exceeds, with a certain probability, a fraction of the competitors terminal wealth. 
Since $T$ is fixed throughout, we delete it from the notation (except for $Z_T$). {\color{black}We also write $\beta_1$ instead of $\beta_{21}$ and $\beta_2$ instead of $\beta_{12}.$} Thus, agent $i$ faces for fixed $X_j, j\neq i$ with $ i,j\in \{1,2\}$, the problem
$$(P_2) \left\{ \begin{array}{l}
\EE \ln (X_i) \to \max\\
\PP(X_i \ge  \beta_j X_j) \ge \alpha_i,\\
\EE[Z_TX_i]=x_0^i,\\
X_i \mbox{ is } \mathcal{F}-\mbox{measurable,}
\end{array}\right.$$
with $\alpha_i,\beta_i\in [0,1]$ for $i=1,2$. The last equation ensures that the wealth $X_i$ has price $x_0^i$, i.e.\ can be attained by a self-financing strategy with initial wealth $x_0^i.$ {\color{black} Throughout we assume that all parameters, preferences and the opportunity investment set are known to all agents.} Now we are seeking for a Nash equilibrium in this situation in terms of terminal wealths. Formally, a Nash equilibrium is here defined as follows.

\begin{definition}\label{definition nash}
Let $J_i(X_1,X_2)= \EE \ln (X_i)$ be the objective function of agent $i$, given that $(X_1,X_2)$ satisfy the constraints in $(P_2)$. A feasible pair $(X_1^*,X_2^*)$ of terminal wealths is called a \emph{Nash equilibrium}, if 
\begin{equation}
J_1(X_1^*,X_2^*)\geq J_1(X_1,X_2^*), \quad J_2(X_1^*,X_2^*)\geq J_2(X_1^*,X_2) \label{Nash condition}
\end{equation}
for all admissible pairs $(X_1,X_2)$ in $(P_2)$.
\end{definition}

\begin{remark}
\begin{itemize}
    \item[a)]  Without the probability constraint, the optimal terminal wealth in $(P_2)$ for agent $i=1,2$ is given by 
    $$ X_1^* = \frac{x_0^1}{Z_T}, \mbox{   and   } X_2^* = \frac{x_0^2}{Z_T}.$$ 
    This result can be found e.g.\ in \cite{korn2013optionsbewertung}.
\item[b)] Once we have the optimal wealths, they can be replicated by  suitable investment strategies due to the completeness of the financial market.
\end{itemize} 
\end{remark}

In what follows, we explicitly determine all Nash equilibria.  Throughout we assume w.l.o.g.\ that $x_0^1\ge x_0^2,$ i.e.\ agent 1 is at least as rich as agent 2. 
{\color{black}
To obtain the Nash equilibria we have to solve the best response problems $(P_2)$ for both agents and then find a fixed point. The best response problems are utility maximization problems with side constraints of the form $X_i \ge Y \mathbbm{1}_A$ with $\PP(A)=\alpha_i$ where the location of the set $A$ varies. Given that $\lambda/Z_T$ is the form of the optimal wealth in the log-utility problem without constraint, it is not too hard to guess that the optimal wealth with constraint satisfies $X_i = \max\{Y, \lambda/Z_T\}$ on a set of probability $\alpha$ and elsewhere takes the  form $\lambda/Z_T$. The crucial question is where the set $A$ is located. We discuss the solution of the best response problem in detail in the appendix. In Lemma \ref{lem:bound2} in Appendix \ref{sec:A1+} we show that the set $A$ is located where it is cheapest to deviate from $\lambda/Z_T$ to satisfy the constraint, i.e. we increase $\lambda$ such that the set $\{\lambda/Z_T \ge Y\}$ has probability $\alpha.$
} We now have to distinguish several parameter cases.

\begin{description}
    \item[Case I: $\alpha_1=\alpha_2=\beta_1=\beta_2=1$] In this case with probability one we must have $X_1^*\ge X_2^*$ and $X_2^*\ge X_1^*,$ hence $X_1^*=X_2^*$. Obviously, this can only be satisfied when $x_0^1 =x_0^2.$ However, it is easy to see that any pair $(X,X)$ of $\mathcal{F}$-measurable random variables with price $\EE [Z_T X]=x_0^1$ constitutes a Nash equilibrium. This is because for given $X$, the probability constraint already determines the terminal wealth of the second agent. There is nothing to optimize here and the shape of the terminal wealth can be arbitrary.   
     \item[Case II: $\alpha_2=1, \alpha_1= 1, \beta_1\beta_2< 1$]
      Using the result in Appendix \ref{sec:A1}, we obtain that the mutual best responses have to be of the form:
      \begin{equation}
          X_1^* = \max\Big\{ \beta_2 X_2^*, \frac{\lambda_1}{Z_T}\Big\}, \quad X_2^* = \max\Big\{ \beta_1 X_1^*, \frac{\lambda_2}{Z_T}\Big\}
      \end{equation}
     for some $\lambda_1,\lambda_2>0.$
     Since $\beta_1\beta_2<1$, this implies 
     \begin{align}
         X_1^* =  \frac{1}{Z_T}  \max\{ \beta_2\lambda_2, \lambda_1\}, \quad
          X_2^* =  \frac{1}{Z_T}  \max\{ \beta_1\lambda_1, \lambda_2\}.
     \end{align}
       Distinguishing the different cases where the maxima are attained and  respecting the self-financing conditions $ \EE[X_1^*Z_T]=x_0^1,  \EE[X_2^*Z_T]=x_0^2,$ we can see that we obtain a solution only in the case of $x_0^2 \ge \beta_1 x_0^1$ and the unique Nash equilibrium (here and elsewhere uniqueness is always up to sets of probability zero) is given by 
     \begin{align}\label{eq:NECaseII0}
         & X_1^* = \frac{x_0^1}{Z_T} \mbox{   and   } X_2^* = \frac{x_0^2}{Z_T}.
     \end{align}  
 \item[Case III: $\alpha_2=1, \alpha_1< 1, \beta_1\beta_2\le 1$]
 Using the result in Appendix \ref{sec:A1}, we obtain that the mutual best responses have to be of the form:
     \begin{align}\label{eq:NoNE1case1}
         X_1^* &= \max\Big\{ \1_{A_1} \beta_2 X_2^*, \frac{\lambda_1}{Z_T}\Big\}, \\
         \label{eq:NoNE2case1}
          X_2^* &= \max\Big\{ \beta_1 X_1^*, \frac{\lambda_2}{Z_T}\Big\},
     \end{align}
     where $\PP(A_1)=\alpha_1$ and $\lambda_1,\lambda_2>0.$ 
     We can plug in \eqref{eq:NoNE2case1} into \eqref{eq:NoNE1case1} to obtain the following expression:
     \begin{align*}
          X_1^* &= \max\Big\{ \1_{A_1} \beta_2  \max\Big\{ \beta_1 X_1^*, \frac{\lambda_2}{Z_T}\Big\}, \frac{\lambda_1}{Z_T}\Big\}. 
     \end{align*}
     Let us first assume that $\beta_1\beta_2<1$. In this case we obtain:
      \begin{align*}
          X_1^* = \frac{1}{Z_T} \Big( \1_{A_1}  \max\{\beta_2 \lambda_2, \lambda_1\}+ \1_{A_1^c} \lambda_1\Big)
     \end{align*}
     and in return
     \begin{align*}
          X_2^* &= \frac{1}{Z_T}   \max\{ \beta_1 \lambda_1, \lambda_2\}.
     \end{align*}
     Distinguishing the different cases where the maxima are attained and  respecting the self-financing conditions $ \EE[X_1^*Z_T]=x_0^1,  \EE[X_2^*Z_T]=x_0^2,$ we can see that we obtain a solution only in the case of $x_0^2 \ge \beta_1 x_0^1$ and the unique Nash equilibrium is given by 
     \begin{align}\label{eq:NECaseII}
         & X_1^* = \frac{x_0^1}{Z_T} \mbox{   and   } X_2^* = \frac{x_0^2}{Z_T}.
     \end{align}  
     Thus, we obtain a  Nash equilibrium only when the wealth of the second agent is not too small. In the case $\beta_1\beta_2=1$ we must have $\PP(X_2\ge X_1)=1$ which implies, since $x_0^1\ge x_0^2$, that $X_1^*=X_2^*$ and necessarily $x_0^1 = x_0^2.$ Thus, the best response to $X$ in $(P_2)$ has to be $X$ again. Now consider Lemma \ref{lem:bound2}. Let $M_\lambda^X := \{ X\le \lambda/Z_T\}$ and suppose that $X^*$ is the best response to $X$. Obviously $X=X^*$ a.s. implies $M_\lambda^X=M_\lambda^{X^*}.$ However, if $\PP(X \neq x_0^1/Z_T)>0$ then $(M_{\lambda_\alpha}^{X^*})^c \cap M_{\lambda_\alpha}^{X}\neq \emptyset$ which is a contradiction.  Hence $X^*=x_0^1/Z_T$ and   the Nash equilibrium is again as in \eqref{eq:NECaseII}.
\item[Case IV: $\alpha_1\le 1, \alpha_2< 1, \beta_1\beta_2\le  1$] This case requires more work. First we note that due to  the result in Appendix \ref{sec:A1}, we obtain that the mutual best responses have to be of the form:
     \begin{align}\label{eq:NoNE1case2}
         X_1^* &= \max\Big\{ \1_{A_1} \beta_2 X_2^*, \frac{\lambda_1}{Z_T}\Big\} = \1_{A_1} \max\Big\{ \beta_2 X_2^*, \frac{\lambda_1}{Z_T}\Big\}+ \1_{A_1^c} \frac{\lambda_1}{Z_T}, \\ 
         \label{eq:NoNE2case2}
          X_2^* &= \max\Big\{  \1_{A_2} \beta_1 X_1^*, \frac{\lambda_2}{Z_T}\Big\}= \1_{A_2} \max\Big\{ \beta_1 X_1^*, \frac{\lambda_2}{Z_T}\Big\}+ \1_{A_2^c} \frac{\lambda_2}{Z_T},
     \end{align}
     where $\PP(A_i)=\alpha_i, i=1,2$ and $\lambda_1,\lambda_2>0.$ 
     Plugging \eqref{eq:NoNE1case2} into \eqref{eq:NoNE2case2} yields:
     \begin{align}
         X_2^* = \max\Big\{  \1_{A_1\cap A_2} \max\Big\{ \beta_1\beta_2 X_2^*, \frac{\beta_1\lambda_1}{Z_T}\Big\}+ \1_{A_1^c\cap A_2} \frac{\lambda_1\beta_1}{Z_T}, \1_{A_2}\frac{\lambda_2}{Z_T}\Big\}+ \1_{A_2^c} \frac{\lambda_2}{Z_T}.
     \end{align}
     We now assume that $\beta_1\beta_2<1.$
     Simplifying this expression, we end up with
     \begin{align}
         X_2^* =  \frac{1}{Z_T} \Big( \lambda_2  \1_{A_2^c}+ \max\{\beta_1\lambda_1,\lambda_2\} \1_{A_2}\Big).
     \end{align}
     This in turn implies 
     \begin{align}
         X_1^* =  \frac{1}{Z_T} \Big(\lambda_1  \1_{A_1^c}+ \max\{\beta_2\lambda_2,\lambda_1\} \1_{A_1}\Big).
     \end{align}
     We also have to respect the self-financing condition which implies the equations
     \begin{align}
         \EE[X_1^*Z_T] &= \lambda_1(1-\alpha_1) +  \max\{\beta_2\lambda_2,\lambda_1\} \alpha_1 = x_0^1,\\
         \EE[X_2^*Z_T] &= \lambda_2(1-\alpha_2) +  \max\{\beta_1\lambda_1,\lambda_2\} \alpha_2 = x_0^2.
     \end{align}
     Depending on where the maxima are attained, we obtain essentially two cases. 

     First consider $\lambda_1\ge \lambda_2\ge \beta_1\lambda_1.$ Here it follows that 
      \begin{align}
         & X_1^* = \frac{x_0^1}{Z_T} \mbox{   and   } X_2^* = \frac{x_0^2}{Z_T},
     \end{align} 
     and that $ x_0^2 \ge \beta_1 x_0^1$ (note that $x_0^1\ge x_0^2$ throughout). The other case which yields a solution is   $ \beta_1\lambda_1\ge \lambda_2.$ In this case
     \begin{align}
         & X_1^* = \frac{x_0^1}{Z_T} \mbox{   and   } X_2^* = \beta_1  \frac{x_0^1}{Z_T} \1_{A_2} +\frac{\lambda_2}{Z_T}\1_{A_2^c}
     \end{align}
     with $\lambda_2= (x_0^2 -\beta_1x_0^1 \alpha_2)/(1-\alpha_2) \in [0,x_0^2].$  This case  occurs when we have $x_0^1 \beta_1 \ge x_0^2 \ge x_0^1 \beta_1 \alpha_2. $ Note that the position of the set $A_2$ is arbitrary i.e.\ we obtain an infinite number of Nash equilibria in this case.

The case $\beta_1\beta_2=1 $ remains. Proceeding as before, we conclude that only the case $\lambda_1\ge \lambda_2$ leads to solutions which are of the form
      \begin{align}
         X_1^* &=  \frac{\lambda_1}{Z_T} \1_{(A_1\cap A_2)^c} +  \max\Big\{X_2^*,\frac{\lambda_1}{Z_T}\Big\} \1_{A_1\cap A_2},\\
         X_2^*  &=  \frac{\lambda_1}{Z_T} \1_{A_2\setminus A_1} +  \max\Big\{X_2^*,\frac{\lambda_1}{Z_T}\Big\} \1_{A_1\cap A_2}+  \frac{\lambda_2}{Z_T} \1_{A_2^c}.
     \end{align}
     By contradiction we obtain that the maximum in the preceding expression has to be $\lambda_1/Z_T.$ As a result $X_1^*=x_0^1/Z_T$ and the solution has the same form as before.
\end{description}

We summarize our findings in the following theorem:

\begin{theorem}[Logarithmic utility] \label{thm:nash_log2}
Let $\alpha_i,\, \beta_i \in [0,1],\, x_0^i > 0,\, i=1,2.$ If $x_0^2 < \alpha_2 \beta_1 x_0^1$, there is no Nash equilibrium. Otherwise, i.e.\ if $x_0^2 \ge  \alpha_2 \beta_1 x_0^1$, there exist three cases. 
\begin{itemize}
    \item[a)] If $\alpha_i = 1, \, \beta_i = 1,\, i=1,2,$ there are infinitely many Nash equilibria of the form $(X,X)$, where $X$ is an $\mathcal{F}$-measurable random variable with $\EE[Z_T X] = x_0^1$. 
    \item[b)] {\color{black} If $\alpha_2 = 1$ and $ \alpha_1 \beta_1 \beta_2 < 1$ } then the unique Nash equilibrium is given by 
    $$ X_1^* = \frac{x_0^1}{Z_T} \mbox{   and   } X_2^* = \frac{x_0^2}{Z_T}. $$
    \item[c)] If $\alpha_1 \leq 1,\, \alpha_2 < 1,$ there are two  cases. If $x_0^2 \geq \beta_1 x_0^1$, the unique Nash equilibrium is given by 
    $$X_1^* = \frac{x_0^1}{Z_T} \mbox{   and   } X_2^* = \frac{x_0^2}{Z_T}.$$
    Otherwise, i.e.\ if $x_0^2 < \beta_1 x_0^1$ there are infinitely many Nash equilibria of the form
    $$ X_1^* = \frac{x_0^1}{Z_T}  \mbox{   and   } X_2^* = \frac{\beta_1 x_0^1}{Z_T} \mathbbm{1}_{A_2} + \frac{x_0^2 - \alpha_2 \beta_1 x_0^1}{(1-\alpha_2)Z_T} \mathbbm{1}_{A_2^c}$$
    with $A_2\in \mathcal{F}$ such that $\PP(A_2) = \alpha_2.$
\end{itemize}
\end{theorem}

\begin{remark}\label{rem:choice_A2}
\begin{itemize}
\item[a)] Note that in case a) of the previous theorem, the condition $x_0^2 \geq \alpha_2 \beta_1 x_0^1$ already implies that $x_0^1 = x_0^2$. In parts b) and c) we see that as long as $x_0^2 \geq \beta_1 x_0^1$, the optimal solution $X_i^* = x_0^i Z_T^{-1}$, $i=1,2,$ to the unconstrained problem is always a Nash equilibrium, independent of $\alpha_i$.  {\color{black} This is because  comparing the $X_i^*$ boils down to comparing the constants and the constraint is obviously   satisfied with probability 1. This also makes sense from an economic point of view, because this is the case where the constraint is already satisfied for the initial wealths and due to the 'no arbitrage' condition it carries over to the terminal wealths. As soon as   $x_0^2 < \beta_1 x_0^1$ we always get different Nash equilibria  since the constraint is satisfied with probability 0 for the unconstraint solutions.} 
    \item[b)] There is a very fundamental difference between Case I and Case II when $\beta_1\beta_1=1$ but $\alpha_1<1$ instead of $\alpha_1=1$. As soon as as there is no strict dominance of the final wealths required, the form $\lambda/Z_T$ of the optimal wealth in the single agent portfolio optimization problem carries through to the Nash equilibrium.
    \item[c)]  Note that in case of a logarithmic utility there can be an infinite number of different Nash equilibria. In particular in Case III, it does not matter where the set $A_2$ exactly is. This is different for other utility functions, see the next section with power utility.  
    \item[d)] It is possible to pick other criteria to choose one of the Nash equilibria in case there are several ones. For example in Case I it makes sense to choose the same Nash equilibrium as in b), since it maximizes the expected utility. In Case IV the expected utility of all Nash equilibria are the same. Hence one could simply look for the Nash equilibrium which maximizes $\E[X_2^*].$ It is easy to see (note that here $\beta_1 x_0^1 > x_0^2$) that this is achieved by choosing $A_2:=\{Z_T\le z_{\alpha_2}\}$ where $z_{\alpha_2}$ is the $\alpha_2-$quantile of $Z_T$, i.e.  $ \PP(Z_T \le z_{\alpha_2})={\alpha_2}$.
\end{itemize}
   
\end{remark}

\subsection{\textcolor{black}{Replicating strategies for the Nash equilibrium}}

\textcolor{black}{Until now, we have only discussed the optimal terminal wealth in the Nash equilibrium. However, since we assumed that the financial market is complete, there exist replicating portfolio strategies for the Nash equilibrium from Theorem~\ref{thm:nash_log2}. We use the Black-Scholes market discussed in Section~\ref{sec:FM} in order to give explicit representations of the replicating strategies. In case b) of Theorem~\ref{thm:nash_log2}, the replicating strategies in terms of fractions of wealth invested into the stock are known to be constant \citep[Mertion ratio, see e.g.][]{merton1969lifetime, korn1997optimal} and given by $(\sigma \sigma^\top)^{-1} \mu$. For later comparison, we consider the process $\pi_2^B$ describing the invested amount of wealth (instead of the fraction), which is given by
$$ \pi_2^B(t) = \frac{x_0^2}{Z_t}\big( \sigma \sigma^\top \big)^{-1} \mu.$$}

\textcolor{black}{Thus, we only discuss the replicating strategy of the second investor in case c) of Theorem~\ref{thm:nash_log2}. Moreover, we cannot give an explicit representation for arbitrary sets $A_2$ and therefore only consider sets of the form $A_2 = \{ c_1 < Z_T < c_2\}$ for constants $0 \leq c_1 < c_2 \leq \infty$ with $\PP(c_1 < Z_T < c_2) = \alpha_2$. Then we can use Lemma~\ref{lem:replicating} from Appendix~\ref{sec:A_strategy} to find the portfolio-wealth pair $(X_2^*(t), \pi_2^*(t))_{t\in [0,T]}$ replicating $X_2^*$. To simplify notation, we introduce the abbreviation
\begin{equation}
    f(c,t) := \frac{\log(c) - \log(Z_t) + \frac{1}{2}\Vert \theta \Vert^2 (T-t)}{\Vert \theta \Vert \sqrt{T-t}},\quad t\in [0,T], \, 0 < c < \infty.
\end{equation}
Here, $(Z_t)_{ t\in [0,T]}$ denotes the state price density process, i.e. 
$$ Z_t = \exp\Big( - \theta^\top W_t - \frac{1}{2} \Vert \theta \Vert^2 t\Big), \, t\in [0,T].$$
}
\textcolor{black}{\begin{theorem}\label{thm:repl_strat}
The replicating portfolio-wealth pair for $X_2^*$ from Theorem~\ref{thm:nash_log2} c) is given by 
\begin{align*}
    X_2^*(t) &= \frac{x_0^2 - \alpha_2 \beta_1 x_0^1}{1-\alpha_2}\Big(Y_1(t) + Y_3(t) \Big) + \beta_1 x_0^1 Y_2(t),\\
     \pi_2^*(t) &=  \frac{x_0^2 - \alpha_2 \beta_1 x_0^1}{1-\alpha_2}\Big(\varphi_1(t) + \varphi_3(t) \Big) + \beta_1 x_0^1 \varphi_2(t), 
\end{align*}
where 
\begin{align*}
    Y_1(t) &= \frac{1}{Z_t} \Phi\big(f(c_1, t) \big),\quad Y_2(t) = \frac{1}{Z_t}\big[ \Phi\big(f(c_2,t) \big) - \Phi\big(f(c_1,t) \big)\big], \quad Y_3(t) = \frac{1}{Z_t} \Phi\big(- f(c_2,t) \big),\\
    \varphi_1(t) &= \Bigg[Y_1(t) + \frac{1}{Z_t \Vert \theta \Vert \sqrt{T-t}} \varphi\big(f(c_1,t)\big)\Bigg] \big( \sigma \sigma^\top\big)^{-1} \mu, \\
    \varphi_2(t) &= \Bigg[Y_2(t) +  \frac{1}{Z_t \Vert \theta \Vert \sqrt{T-t}} \big( \varphi\big(f(c_2,t) \big) - \varphi\big(f(c_1,t) \big)\big)\Bigg]\big( \sigma \sigma^\top\big)^{-1} \mu,\\
    \varphi_3(t) &= \Bigg[Y_3(t) - \frac{1}{Z_t\Vert \theta \Vert \sqrt{T-t}} \varphi\big( f(c_2,t)\big) \Bigg] \big( \sigma \sigma^\top)^{-1} \mu.
\end{align*}
Here the portfolio strategy $\pi_2^*$ denotes the amount of wealth invested into the stocks at time $t\in [0,T).$ 
\end{theorem}}
\textcolor{black}{\begin{proof}
    If we write $X_2^*$ from Theorem~\ref{thm:nash_log2} c) for $A_2=\{c_1 < Z_T < c_2\}$ as 
    $$ X_2^* = \frac{x_0^2 - \alpha_2 \beta_1 x_0^1}{1-\alpha_2} Z_T^{-1}\left(\mathbbm{1}\{Z_T \leq c_1\} + \mathbbm{1}\{Z_T \geq c_2\} \right)  + \beta_1 x_0^1 Z_T^{-1} \mathbbm{1}\{c_1 < Z_T < c_2\},$$ the result follows from Lemma~\ref{lem:replicating} and the linearity of the wealth process. 
\end{proof}}

\subsection{Power utility}
We consider now the same problem with power utility for the agents. Note that we use the same parameter for the utility function for both agents. More precisely, we consider for $U(x)= \frac{1}{1-\gamma} x^{1-\gamma},\,  {\color{black}\gamma >0 ,\gamma \neq 1},$ the individual portfolio problems
$$(P_2) \left\{ \begin{array}{l}
\EE U(X_i) \to \max\\
\PP(X_i \ge \sum_{j\neq i} \beta_j X_j) \ge \alpha_i,\\
\EE[Z_TX_i]=x_0^i,\\
X_i \mbox{ is } \mathcal{F}-\mbox{measurable,}
\end{array}\right.$$
for $i=1,2$. For notational convenience, we abbreviate $\varepsilon_\gamma := \EE\big[Z_T^{1-1/\gamma} \big].$

As in the previous case of logarithmic utility, it is easy to see that for $x_0^1 \ge \beta_2 x_0^2 \ge \beta_1\beta_2 x_0^1$ the optimal terminal wealths without constraint constitute a Nash equilibrium, i.e. 
$$ X_1^* =  \frac{x_0^1}{\varepsilon_\gamma} Z_T^{-1/\gamma}, \mbox{   and   } X_2^* = \frac{x_0^2}{\varepsilon_\gamma}  Z_T^{-1/\gamma},$$
since $I(x):= x^{-1/\gamma}$ is the inverse function of $U'$ \citep[see e.g.][]{korn2013optionsbewertung}.

We concentrate now on the interesting case which is determined by the parameters $\beta_1\beta_2<1$ and {\color{black} for $\gamma \in(0,1)$} by
\begin{equation}
  \beta_1 x_0^1 > x_0^2\ge \frac{\beta_1 x_0^1}{\varepsilon_\gamma}  \int_{Z_T\ge z_{1-\alpha_2}} Z_T^{1-1/\gamma}d\PP,  \label{eq:case2}
\end{equation}
and  {\color{black} for $\gamma >1$ by 
\begin{equation}
  \beta_1 x_0^1 > x_0^2\ge \frac{\beta_1 x_0^1}{\varepsilon_\gamma}  \int_{Z_T\le z_{\alpha_2}} Z_T^{1-1/\gamma}d\PP,  \label{eq:case3}
\end{equation}
}
where $z_{\alpha}$ is the $\alpha-$quantile of $Z_T$, i.e.  $ \PP(Z_T \le z_\alpha)=\alpha$. Since $Z_T>0$ we necessarily have that $z_\alpha> 0.$ 

\begin{theorem}[Power utility]\label{lem:power}
     Let $\beta_1\beta_2<1$.
\begin{itemize}
    \item[a)] If $\gamma \in(0,1)$  and the inequality in \eqref{eq:case2} holds true,  the unique Nash equilibrium is given by
     \begin{equation}\label{eq:NEpower2} 
X_1^* = \frac{x_0^1}{\varepsilon_\gamma} Z_T^{-1/\gamma}, \mbox{   and   } X_2^* = \left\{ \begin{array}{cr}  (\lambda_2Z_T)^{-1/\gamma},& Z_T < z_{1-\alpha_2},\\ \beta_1 \frac{x_0^1}{\varepsilon_\gamma} Z_T^{-1/\gamma},& Z_T \ge  z_{1-\alpha_2}, \end{array}\right. \end{equation} 
     where  $\lambda_2>0$ is such that $\EE[Z_T X_2^*]=x_0^2$.
    \item[b)] {\color{black} If $\gamma>1$  and the inequality in \eqref{eq:case3} holds true,  the unique Nash equilibrium is given by
     \begin{equation}\label{eq:NEpower3} 
X_1^* = \frac{x_0^1}{\varepsilon_\gamma} Z_T^{-1/\gamma}, \mbox{   and   } X_2^* = \left\{ \begin{array}{cr}  (\lambda_2Z_T)^{-1/\gamma},& Z_T > z_{\alpha_2},\\ \beta_1 \frac{x_0^1}{\varepsilon_\gamma} Z_T^{-1/\gamma},& Z_T \le   z_{\alpha_2}, \end{array}\right. \end{equation} 
     where  $\lambda_2>0$ is such that $\EE[Z_T X_2^*]=x_0^2$.}
\end{itemize}     
    
\end{theorem}
\begin{proof} We restrict the proof to case a). Case b) is similar.
Using Lemma \ref{lem:bound} and rearranging the terms, we obtain that a Nash equilibrium is necessarily of the form
 \begin{align*}
  X_1^* &=  {Z_T}^{-1/\gamma} \Big( \lambda_1  \1_{A_1^c}+ \max\{\beta_2\lambda_2,\lambda_1\} \1_{A_1}\Big),\\
         X_2^* &=  {Z_T}^{-1/\gamma} \Big( \lambda_2  \1_{A_2^c}+ \max\{\beta_1\lambda_1,\lambda_2\} \1_{A_2}\Big)
     \end{align*}
     for some $\lambda_1,\lambda_2 >0$ and measurable sets $A_1,A_2$ with probability $\PP(A_i)=\alpha_1, i=1,2.$
Distinguishing the cases where the maximum is attained, it is possible to see that only the case $\lambda_1 \ge \beta_2\lambda_2$ yields a solution. Moreover, it follows that
$$X_1^* = \frac{x_0^1}{\varepsilon_\gamma} Z_T^{-1/\gamma}.$$
In order to determine the precise form of $X_2^*$,  we will proceed in a different way as in the logarithmic case. This is necessary, because in the power setting it will turn out that it really matters where the region is located precisely where the probability constraint is satisfied. This does not follow from Lemma \ref{lem:bound}. \\
We write $\kappa := U'(x_0^1/\varepsilon_\gamma)$, i.e. $I(\kappa) = x_0^1/\varepsilon_\gamma.$ It follows that $I(\lambda_2) < \beta_1 I(\kappa)$, because otherwise the price of $X_2^*$ would be larger or equal than $\beta_1 x_0^1$ which contradicts our assumption. This implies $\lambda_2 > \kappa \beta_1^{-\gamma}$. 
It remains to prove that $X_2^*$ is the best response to $X_1^*$. Thus, we consider the following function
     $$ L(X,\lambda_2,\eta_2) := U(X) -\lambda_2 Z_T X +\eta_2 \1\{X\ge \beta_1 I(\kappa Z_T)\},$$
     where
     $$ \eta_2 := U(I(\lambda_2z_{1-\alpha_2}))- U(\beta_1I(\kappa z_{1-\alpha_2}))+ \lambda_2 z_{1-\alpha_2} \Big( \beta_1I(\kappa z_{1-\alpha_2})- I( \lambda_2 z_{1-\alpha_2})\Big) $$
and $\kappa,\lambda_2$ as before. Note that $\eta_2\ge 0$ since
\begin{align*}
    \eta_2\ge 0 & \Leftrightarrow U(I(\lambda_2z_{1-\alpha_2}))- U(\beta_1I(\kappa z_{1-\alpha_2})) \ge \lambda_2 z_{1-\alpha_2} \Big( I( \lambda_2 z_{1-\alpha_2})- \beta_1I(\kappa z_{1-\alpha_2}) \Big)\\
    &  \Leftrightarrow \frac{U(I(\lambda_2z_{1-\alpha_2}))- U(\beta_1I(\kappa z_{1-\alpha_2}))}{I( \lambda_2 z_{1-\alpha_2})- \beta_1I(\kappa z_{1-\alpha_2}) } \le \lambda_2 z_{1-\alpha_2}\\ &  \Leftrightarrow  U'(\xi) \le \lambda_2 z_{1-\alpha_2}
\end{align*}
for some $\xi \in (I(\lambda_2z_{1-\alpha_2}), \beta_1 I(\kappa z_{1-\alpha_2})).$ Because $U'$ is non-increasing, we obtain $ U'(\xi) \le U'(I(\lambda_2z_{1-\alpha_2}))=\lambda_2z_{1-\alpha_2}.$
Now we can show that for any other admissible terminal wealth $X_2$ we obtain
\begin{align*}
    \EE U(X_2^*) - \EE U(X_2) \geq \EE L(X_2^*,\lambda_2,\eta_2)- \EE L(X_2,\lambda_2,\eta_2)\ge 0,
\end{align*}
since $X_2^*$ maximizes $X\mapsto L(X,\lambda_2,\eta_2)$ pointwise (see App. \ref{sec:A2}) which implies that $X_2^*$ maximizes the objective of agent 2 under the VaR-based constraint. Now it only remains to discuss the existence of $\lambda_2>0$ such that $\E[Z_T X_2^*] = x_0^2$. We can solve
$$\E[Z_T X_2^*] = \lambda_2^{-1/\gamma} \E\Big[Z_T^{1-1/\gamma} \1 \{Z_T< z_{1-\alpha_2}\} \Big] + \frac{\beta_1 x_0^1}{\varepsilon_\gamma} \E\Big[Z_T ^{1-1/\gamma} \1 \{Z_T \geq z_{1-\alpha_2}\} \Big] = x_0^2$$ 
for $\lambda_2$ to obtain 
$$ \lambda_2 = \E\Big[Z_T^{1-1/\gamma} \1 \{Z_T< z_{1-\alpha_2}\} \Big]^{\gamma} \bigg(x_0^2 -  \frac{\beta_1 x_0^1}{\varepsilon_\gamma} \E\Big[Z_T ^{1-1/\gamma} \1 \{Z_T \geq z_{1-\alpha_2}\} \Big]\bigg)^{-\gamma}.$$
Since \eqref{eq:case2} holds, we have $\lambda_2 > 0$ which concludes the proof.
\end{proof}

\begin{remark}
\begin{enumerate}
    \item[a)] Note that for power utility, in contrast to logarithmic utility, the position of the set $A$ that guarantees that the probability constraint is satisfied is not arbitrary. This is remarkable since the limit $\gamma \to 1$ results in a logarithmic utility function. {\color{black} Moreover, it turns out that for $\gamma >1$ the constraint is satisfied on the set $\{Z_T \leq z_{\alpha_2}\}$ whereas for $\gamma\in(0,1)$, the constraint is satisfied on the set $\{Z_T > z_{1-\alpha_2}\}$. In both cases the constraint is satisfied on the side where it is cheaper to achieve. This is the main reason for the non-continuous behavior of the optimal terminal wealth as a function of $\gamma.$ It cannot  be explained by the risk aversion.  The relative risk aversion of the power utility is given by $-x\frac{U''(x)}{U'(x)}=\gamma$ and thus risk aversion is increasing in $\gamma$. When we consider as an example again the Black Scholes market, it holds that large $Z_T$ correspond to small stock prices and vice versa.  This means that the less risk-averse investor (less risk averse than log-utility) tries to outperform the other agent  in the event of small stock prices, whereas the more risk-averse investor (more risk averse than log-utility) tries to outperform the other agent in the event of large stock prices.}
     \item[b)] \textcolor{black}{It is possible to derive the replicating strategies for the Nash equilibrium from Theorem~\ref{lem:power} using Theorem~E.1 from \cite{jin2008behavioral}. However, the resulting expressions become even more complex than the ones obtained for logarithmic utility which is why we refrain from presenting them here.}
\end{enumerate}
    
\end{remark}


\section{Multi agent case}\label{sec:multi}
Here we consider problem $(P_2)$ in the case of $n\ge 3$ agents, i.e. for $i=1,\ldots,n$ we look at
$$(P_n) \left\{ \begin{array}{l}
\EE \ln (X_i) \to \max\\
\PP(X_i \ge \sum_{j\neq i} \beta_{{\color{black}ij}} X_j) \ge \alpha_i,\\
\EE[Z_TX_i]=x_0^i,\\
X_i \mbox{ is } \mathcal{F}-\mbox{measurable,}
\end{array}\right.$$
with $\alpha_i,\beta_{{\color{black}ij}}\in [0,1]$.  Further, we assume w.l.o.g.\ that the initial capitals of the agents are ordered by $x_0^1\ge \ldots \ge x_0^n.$ In particular, we restrict the discussion to the logarithmic utility. Moreover, we assume that for all $i$ we have $\sum_{j\neq i} \beta_{{\color{black}ij}} \le 1.$ We are again looking for a Nash equilibrium of investment strategies. The definition for $n$ agents is as follows:

\begin{definition}\label{definition nash2}
Let $J_i(X_1,\ldots,X_n)= \EE \ln (X_i)$ be the objective function of agent $i$ given $(X_1,\ldots ,X_n)$ satisfy the constraints in $(P_n)$. A feasible vector $(X_1^*,\ldots,X_n^*)$ of terminal wealths is called a \emph{Nash equilibrium}, if, for all admissible random vectors $(X_1^*,\ldots,X_{i-1}^*,X_i,X_{i+1}^*,\ldots,X_n^*)$ on the right-hand side: 
\begin{equation}
J_i(X_1^*,\ldots,X_n^*)\geq J_i(X_1^*,\ldots,X_{i-1}^*,X_i,X_{i+1}^*,\ldots,X_n^*) \label{Nash condition2}
\end{equation}
for all agents $i.$
\end{definition}

We distinguish now the following cases:

\subsection{Assume that the sum of the alphas is less or equal to 1}
In this case, the probability constraints of the agents can (and will) be satisfied on disjoint sets. 

\begin{theorem}\label{thm:nash_logn}
    If a Nash equilibrium exists, it is of the form
    \begin{equation}\label{eq:Xstardisjoint}
        X_i^* = \frac{1}{Z_T} \Big(\1_{A_i} \max\{x_0^i, \lambda_\beta^{-i}\}+\1_{A_i^c} \lambda_i \Big)
    \end{equation} 
    with $\lambda_\beta^{-i}:= \sum_{j\neq i} \beta_{{\color{black}ij}} \lambda_j$, $\PP(A_i)=\alpha_i$ and $A_i\cap A_j=\emptyset$ for $i\neq j.$ Moreover, we have that \begin{equation}\label{eq:lambda}
        \lambda_i = \frac{1}{1-\alpha_i} \Big(x_0^i -\alpha_i \max\{x_0^i, \lambda_\beta^{-i}\} \Big)
    \end{equation} 
     and $0<\lambda_i\le x_0^i$, $\alpha_i \lambda_\beta^{-i} \le x_0^i.$
\end{theorem}

\begin{proof}
In order to determine a Nash equilibrium, we have to solve the best response problem $(P_n)$ for an arbitrary  agent $i$. Suppose that $X_j, j\neq i,$ are arbitrary given wealths of agents $j$. We can reformulate problem $(P_n)$ as follows:
$$(P_n) \left\{ \begin{array}{l}
\EE \ln (X_i) \to \max\\
X_i \ge \1_{A_i} \sum_{j\neq i} \beta_{{\color{black}ij}} X_j,\\
\EE[Z_TX_i]=x_0^i,\\
X_i \mbox{ is } \mathcal{F}-\mbox{measurable,}\\
A_i \in \mathcal{F}, \PP(A_i)=\alpha_i.
\end{array}\right.$$
Note that the optimization is over $X_i$ and the set $A_i$ here. According to Lemma \ref{lem:bound}, an optimal solution is of the form
\begin{eqnarray*}
    X_i &= \max\left\{\1_{A_i} X_\beta^{-i}, \frac{\lambda_i}{Z_T}  \right\} =
    \1_{A_i} \max\left\{ X_\beta^{-i}, \frac{\lambda_i}{Z_T}  \right\} + \1_{A_i^c} \frac{\lambda_i}{Z_T},
\end{eqnarray*}
where $X_\beta^{-i}:= \sum_{j\neq i} \beta_{{\color{black}ij}} X_j$. The maximum construction and the self-financing constraint yield $X_i\ge \lambda_i/Z_T$ and $\lambda_i\le x_0^i.$ Thus, we obtain that the minimal value of $Z_T X_\beta^{-i}$ is attained on the set $ (\cup_{j\neq i} A_j)^c.$ Taking Lemma \ref{lem:bound2} into account, in order to maximize $\EE \ln (X_i)$ we have to choose $A_i$ such that $A_i\cap A_j=\emptyset$ for $i\neq j.$ Due to the assumption $\alpha_1+\ldots +\alpha_n\le 1$, this is possible and the sets $A_i$ are disjoint in a Nash equilibrium.
In particular we obtain
\begin{eqnarray*}
    X_1 & =& \1_{A_1} \max\left\{ X_\beta^{-1}, \frac{\lambda_1}{Z_T}  \right\} + \1_{A_1^c} \frac{\lambda_1}{Z_T},\\
    X_2 & =& \1_{A_2} \max\bigg\{ \beta_{{\color{black}21}} X_1+ \sum_{j\ge 3} \beta_{{\color{black}2j}} X_j, \frac{\lambda_2}{Z_T}  \bigg\} + \1_{A_2^c} \frac{\lambda_2}{Z_T}.
\end{eqnarray*}
Plugging in $X_1$ into $X_2$ yields (note that $A_1$ and $A_2$ are disjoint)
\begin{equation}
     X_2 = \1_{A_2} \max\bigg\{ \frac{\beta_{{\color{black}21}} \lambda_1}{Z_T}+ \sum_{j\ge 3} \beta_{{\color{black}2j}} X_j, \frac{\lambda_2}{Z_T}  \bigg\} + \1_{A_2^c} \frac{\lambda_2}{Z_T}.
\end{equation}
Continuing this procedure we finally obtain:
\begin{equation}
     X_i = \frac{1}{Z_T} \Big( \1_{A_i} \max\left\{\lambda_\beta^{-i},\lambda_i \right\} + \1_{A_i^c} \lambda_i\Big) .
\end{equation}

If $\lambda_i \ge \lambda_\beta^{-i}$, then $X_i=\lambda_i/Z_T$ and by the financing condition $\lambda_i=x_0^i.$ Thus, in the maximum we can replace $\lambda_i$ by $x_0^i.$ Hence,  $X_i^*$ is as stated in \eqref{eq:Xstardisjoint}. The parameter $\lambda_i$ has then to be chosen such that the wealth can be financed with initial capital $x_0^i$. Hence
    $$ \EE [Z_T X_i^*] = \alpha_i\max\{x_0^i,\lambda_\beta^{-i} \} +(1-\alpha_i)  \lambda_i=x_0^i.$$ Solving this equation for $\lambda_i$ yields \eqref{eq:lambda}. In order to have a viable solution, we must have $\lambda_i\ge 0$ which yields the inequalities.

\end{proof}

\begin{remark}
\begin{itemize}
    \item[a)] Obviously, there can be infinitely many different Nash equilibria depending on where precisely the sets $A_i$ are located.
    \item[b)] Since $\lambda_i \le x_0^i$ and $\sum_{j\neq i} \beta_{{\color{black}ij}} \le 1$ we always have that $x_0^1 \ge \Bar{\lambda}_{-1}$ and thus $X_1^*=  \frac{x_0^1}{Z_T}$. Hence, the richest agent will never be influenced by the probability constraint.
\end{itemize}   
\end{remark}

\subsection{Assume that the sum of the alphas is larger than 1} In this case, the probability constraints of the agents obviously cannot be satisfied on disjoint sets. Inspired by the previous subsection we will determine only those Nash equilibria which are of the form
 \begin{align*}
        X_i &= \frac{1}{Z_T}\sum_{k=1}^m \lambda_{ki} \1_{B_k}
    \end{align*}
    for some sets $B_k$ and constants $\lambda_{ki}.$
    Further, we only consider the case where the probabilities $\alpha_i$ in the constraint satisfy $\alpha_i = \ell_i h$ with $h=1/m>0$ and $\ell_i\in\N.$ Note that this is always satisfied when $\alpha_i\in \mathbb{Q}$, hence not really restrictive.  Let $(B_i)_{i=1,\ldots,m}$ be a partition of $\Omega,$ i.e.\ $\cup_i B_i=\Omega$ and $B_i \cap B_j=\emptyset$ for $i\neq j$ and $\PP(B_i)=h.$ Consider the following deterministic optimization problems for $i=1,\ldots,n$:
$$(PD_i) \left\{ \begin{array}{l}
 \sum_{k=1}^m  \ln(\lambda_{ki}) \to \max\\
\lambda_{ki} \ge M_{ki} \sum_{j\neq i} \beta_{{\color{black}ij}} \lambda_{kj}, \;\mbox{ for all }k,\\
\sum_{k=1}^m \lambda_{ki}= m x_0^i, \\
\sum_{k=1}^m M_{ki} = \ell_i,\\
\lambda_{ki}>0, M_{ki}\in \{0,1\}, \;\mbox{ for all }k.
\end{array}\right.$$
It turns out that  $(PD_i)$ is the best response problem for agent $i$.

\begin{theorem}\label{thm:nash_log_n_groesser}
Suppose there exist $(\lambda_{kj})$ and $(M_{kj})$, $j=1,\ldots,n,\, k=1,\ldots,m,$ such that $(\lambda_{ki})$ and $(M_{ki})$, $k=1,\ldots,m,$ yield an optimal solution to $(PD_i)$ for each $i\in \{1,\ldots,n\}.$ Then the terminal wealths 
    \begin{align}\label{eq:Xwealth}
        X_i^* &= \frac{1}{Z_T}\sum_{k=1}^m \lambda_{ki} \1_{B_k}
    \end{align}
    constitute a Nash equilibrium for problem $(P_n).$
\end{theorem}

\begin{proof}
    Suppose the $X_i$ are as stated. When we fix $X_j, j\neq i,$ we have to show that $X_i$ is the best response, i.e.\ solves problem $(P_n).$ By Lemma \ref{lem:bound2}, we know that the best response is given by
    $$ X_i = \frac{1}{Z_T} \Big( \sum_{k\in S} \max\big\{ \sum_{j\neq i} \beta_{{\color{black}ij}} \lambda_{kj}, \lambda\big\}  \1_{B_k}+  \sum_{k\in S^c}  \lambda  \1_{B_k}\Big)$$
    where $S$ is a set of $\ell_i$ indices where $ \sum_{j\neq i} \beta_{{\color{black}ij}} \lambda_{kj}$ is smallest over all $k.$ Thus, we can write $X_i$ as in \eqref{eq:Xwealth} with $\lambda_{ki}$ either being $\sum_{j\neq i} \beta_{{\color{black}ij}} \lambda_{kj}$ or $\lambda.$ The self-financing condition can be stated as $\EE[Z_T X_i] = 1/m \sum_{k=1}^m \lambda_{ki}=  x_0^i.$ The probability constraint is satisfied when for $\ell_i$ indices $\lambda_{ki} \ge \sum_{j\neq i} \beta_{{\color{black}ij}} \lambda_{kj}.$ Since $\ln(x)$ is Schur-concave the optimal indices are automatically chosen by the optimization problem.
\end{proof}

\section{Numerical Examples}
Let us discuss the Nash equilibria from Theorems \ref{thm:nash_log2}, \ref{lem:power} and \ref{thm:nash_logn} by considering some numerical examples. To do this, we use a Black-Scholes financial market consisting of one stock with price process 
\begin{equation}
    \mathrm{d}S_t = S_t \big( \mu \mathrm{d}t + \sigma \mathrm{d}W_t\big),\, t\in [0,T],\quad S_0 = 1
\end{equation}
and a riskless bond with zero interest rate. We set the market parameters to $T=4$, $\mu = 0.03$, and $\sigma = 0.2$. In this case, the state price density $Z_T$ follows a lognormal distribution, i.e. 
\begin{equation}
    Z_T \sim \mathrm{LN}\Big(-\frac{\mu^2}{2\sigma^2}T, \frac{\mu^2}{\sigma^2}T\Big) =: \mathrm{LN}\big(\nu,\tau^2 \big)
\end{equation} 
with $\nu = -0.045$ and $\tau^2 = 0.09.$ {\color{black} It is important to note here that large $Z_T$ correspond to small stock prices and vice versa.}

\subsection{\textcolor{black}{Two agent case with logarithmic utility}}

First, we consider the 2-agent equilibrium under logarithmic utilities from Theorem \ref{thm:nash_log2}. We choose $x_0^1 = 3$ and $x_0^2 = 2$. Figure \ref{fig:nash_log2} shows the Nash equilibrium $(X_1^*, X_2^*)$ as a function of the state-price density for $\alpha_2 = 0.2$ and $\beta_1 = 0.9$, i.e. $\alpha_2 \beta_1 x_0^1 \leq x_0^2 < \beta_1 x_0^1$, so that the Nash equilibrium is given in the second part of Theorem \ref{thm:nash_log2} c). Note that $\alpha_1$ and $\beta_2$ have no influence on the Nash equilibrium. The set $A_2$ is chosen as $A_2 = \{Z_T \leq z_{\alpha_2}\}$ as discussed in Remark~\ref{rem:choice_A2}. The purple and green solid lines show the terminal wealth of agent 1 and 2 in the Nash equilibrium. The black dashed line shows for comparison the optimal terminal wealth of agent 2 in the standard Merton problem without the VaR-based constraint. The terminal wealth $X_1^*$ is continuous, strictly decreasing, and strictly convex in terms of $Z_T$ while $X_2^*$ shows a similar overall behavior with a discontinuity located at $z_{\alpha_2}$. We notice, that the terminal wealth $X_2^*$ in the Nash equilibrium is larger than the standard solution if $Z_T \leq z_{\alpha_2}$ and smaller for $Z_T > z_{\alpha_2}$. {\color{black} This means that the agent insures the constraint in scenarios with a bullish stock price evolution and underperforms in bearish markets. However, as discussed before, the insured scenarios can to some extend be arbitrarily defined in the case of an agent with log-utility, as long as the scenario set has the desired probability.  }

\begin{figure}
    \vspace*{\fill}
    \centering
    \includegraphics[width=\linewidth]{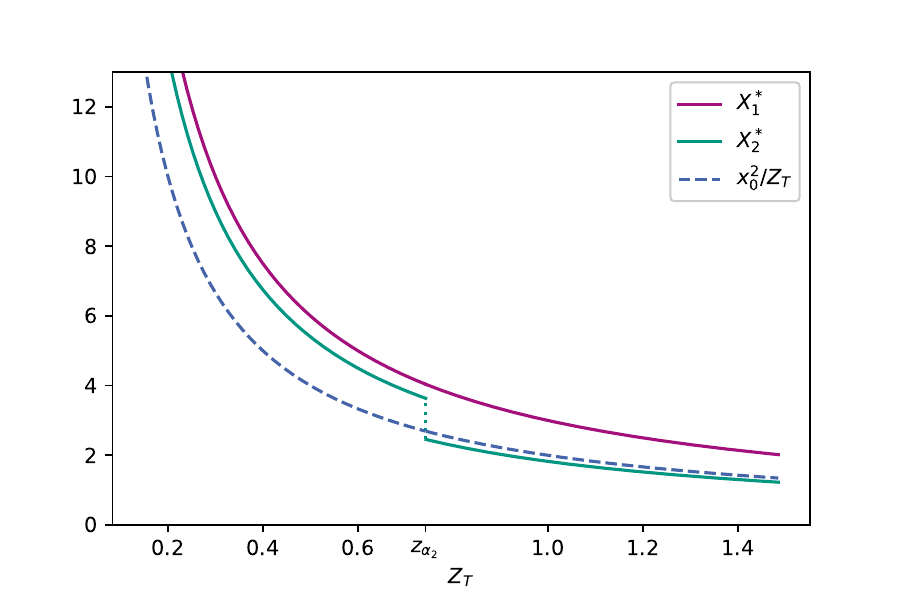}
    \caption{Nash equilibrium $(X_1^*, X_2^*)$ from Theorem \ref{thm:nash_log2} in terms of $Z_T$ for $\alpha_2 = 0.2$, $\beta_1 = 0.9$.}
    \label{fig:nash_log2}
\end{figure}

\begin{figure}
     \centering
     \vspace*{\fill}
     \begin{subfigure}[b]{0.468\textwidth}
         \centering
         \includegraphics[width=\textwidth]{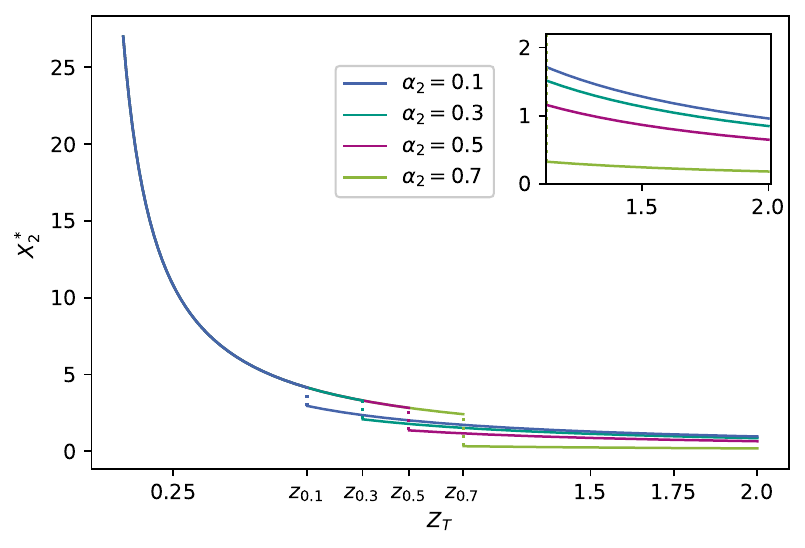}
         \caption{$\beta_1 = 0.9$ and $\alpha_2 \in \{0.1,0.3,0.5,0.7\}$}
    \label{fig:nash_log_influence_alpha}
     \end{subfigure}
     \hfill
     \begin{subfigure}[b]{0.48\textwidth}
         \centering
         \includegraphics[width=\textwidth]{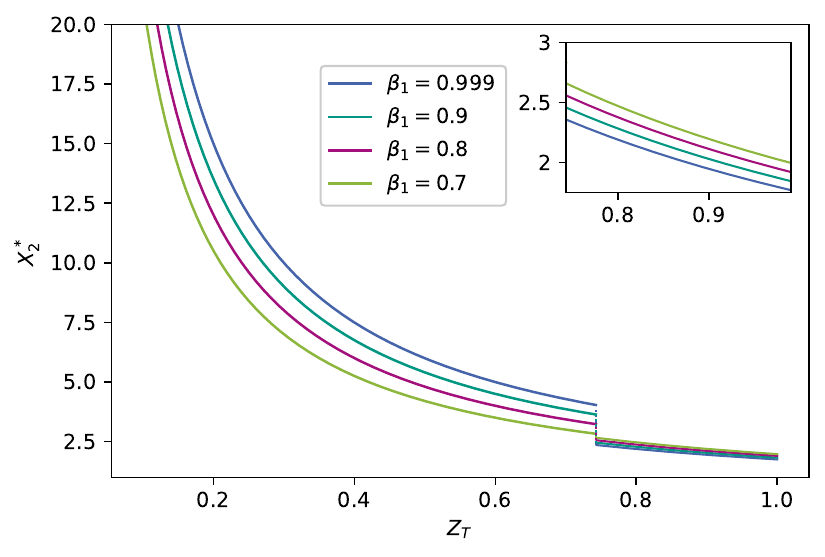}
         \caption{$\alpha_2 = 0.2$ and $\beta_1 \in \{0.7, 0.8, 0.9, 0.999\}$}
\label{fig:nash_log_influence_beta}
     \end{subfigure}
     \caption{Terminal wealth $X_2^*$ of agent 2 in the Nash equilibrium from Theorem \ref{thm:nash_log2} in terms of $Z_T$ for varying values of $\alpha_2$ and $\beta_1$.}
\end{figure}

Figure \ref{fig:nash_log_influence_alpha} illustrates the influence of the parameter $\alpha_2$ on the terminal wealth $X_2^*$ of agent 2 in the Nash equilibrium. As expected, the location of the discontinuity of $X_2^*$ is increasing in terms of $\alpha_2$ as it is simply given as the $\alpha_2$-quantile of $Z_T$. Moreover, we notice that for $Z_T > z_{\alpha_2}$ the value of $X_2^*$ is decreasing in $\alpha_2$. This results from the budget constraint $\E[Z_T X_2^*] = x_0^2$ since for larger $\alpha_2$, the value of $X_2^*$ is kept on the larger value $\beta_1 X_1^*$ on a larger interval. 


Figure \ref{fig:nash_log_influence_beta} shows the terminal wealth $X_2^*$ of agent 2 in the Nash equilibrium for different choices of the parameter $\beta_1$. We notice a change of order of the value of $X_2^*$ for the different choices of $\beta_1$ located at the discontinuity $z_{\alpha_2}.$ For $Z_T \leq z_{\alpha_2}$, the value of $X_2^*$ is largest for the largest choice of $\beta_1$ as the value of $X_2^*$ is kept at $\beta_1 X_1^*$ in this case. For $Z_T > z_{\alpha_2}$, the order of the values changes, i.e. the largest choice of $\beta_1$ yields the smallest value of $X_2^*$. This is again due to the budget constraint $\E[Z_T X_2^*] = x_0^2$.  {\color{black} It is interesting to see here that the set of insured scenarios is not influenced by $\beta_1.$ This parameter only has an influence on the severity of underperformance in non-insured states.   }


Next, let us consider some different possible choices for the set $A_2$. As discussed in Remark~\ref{rem:choice_A2}, there are infinitely many possible choices for $A_2$. Although choosing $A_2 = \{Z_T \leq z_{\alpha_2}\}$ maximizes the expected terminal wealth, it is worth considering different choices for $A_2$ and their influence on the terminal wealth $X_2^*$ of agent 2 in the Nash equilibrium. Figure \ref{fig:diff_sets} shows $X_2^*$ for different choices of $A_2$ in the form $A_2 = \{c_1 \leq Z_T \leq c_2\}$. Since $Z_T \sim \mathrm{LN}(\nu, \tau)$, for a fixed lower bound $c_1>0$, the upper bound $c_2>0$ of the interval is determined via 
\begin{equation}
    c_2 = \exp\bigg(\nu + \tau \cdot \Phi^{-1}\bigg(\alpha_2 + \Phi\bigg(\frac{\mathrm{ln}(c_1) - \nu}{\tau} \bigg) \bigg) \bigg). \label{eq:upper_bound}
\end{equation}
Note that the largest possible choice for the lower bound $c_1$ is the $(1-\alpha_2)$-quantile of $Z_T$, i.e. 
$$ z_{1-\alpha_2} = \exp\left(\nu + \tau \cdot \Phi^{-1}(1-\alpha_2)\right) \approx 1.2306.$$

We notice that the length of the interval differs significantly depending on whether the interval is located left, right, or around the mode of the distribution of $Z_T$.

\begin{figure}
    \begin{minipage}{\linewidth}
    \vspace{5cm}
     \centering
     \begin{subfigure}[b]{0.49\textwidth}
         \centering
         \includegraphics[width=\textwidth]{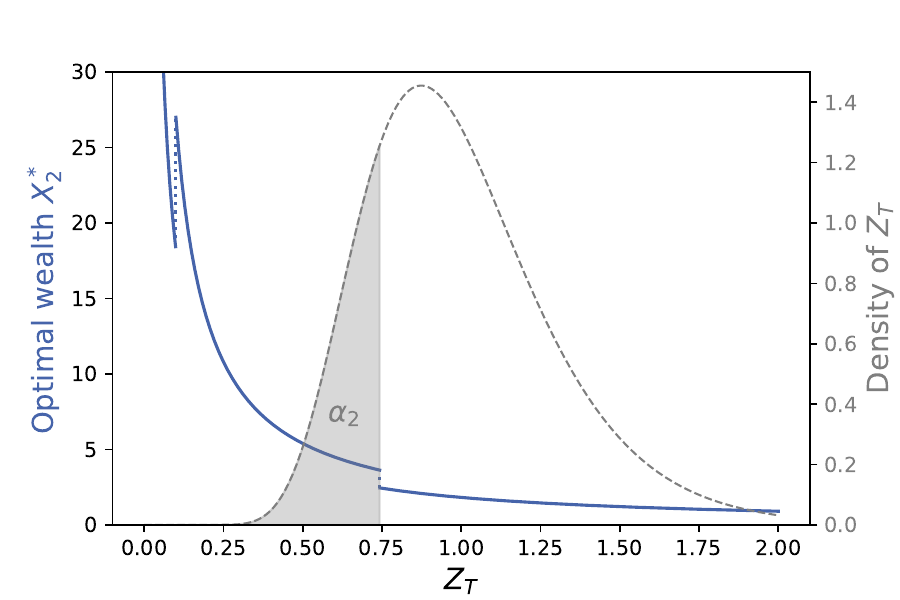}
         \caption{$A_2 = [0.1, 0.7427]$}
         \label{fig:y equals x}
     \end{subfigure}
     \hfill
     \begin{subfigure}[b]{0.49\textwidth}
         \centering
         \includegraphics[width=\textwidth]{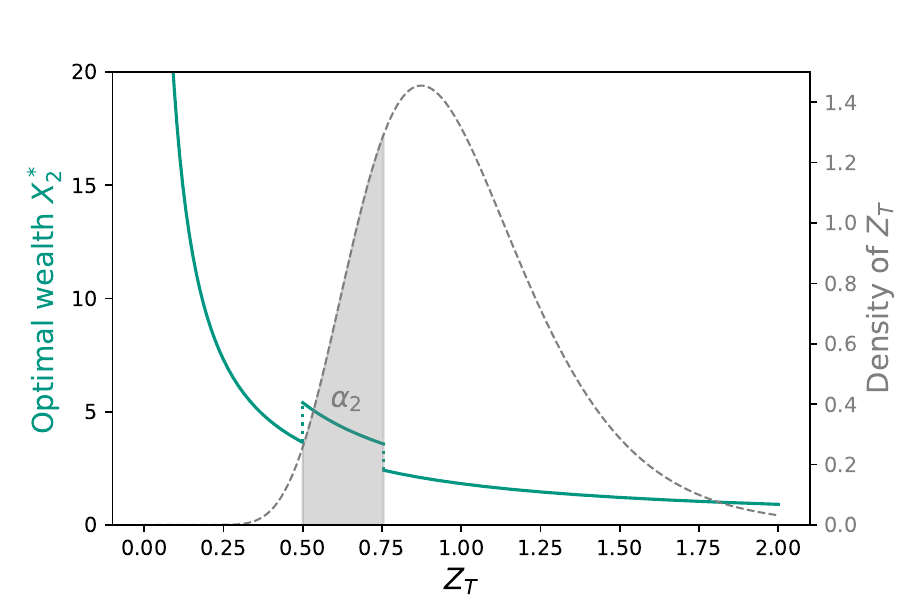}
         \caption{$A_2 = [0.5, 0.7547]$}
         \label{fig:three sin x}
     \end{subfigure}
     \\
     \begin{subfigure}[b]{0.49\textwidth}
         \centering
         \includegraphics[width=\textwidth]{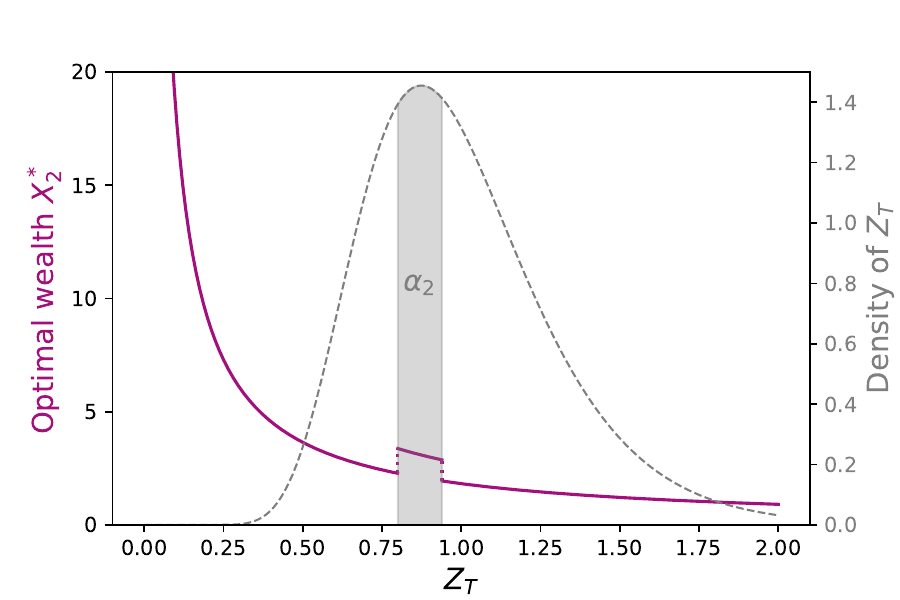}
         \caption{$A_2 = [0.8, 0.9391]$}
         \label{fig:five over x}
     \end{subfigure}
     \hfill
     \begin{subfigure}[b]{0.49\textwidth}
         \centering
         \includegraphics[width=\textwidth]{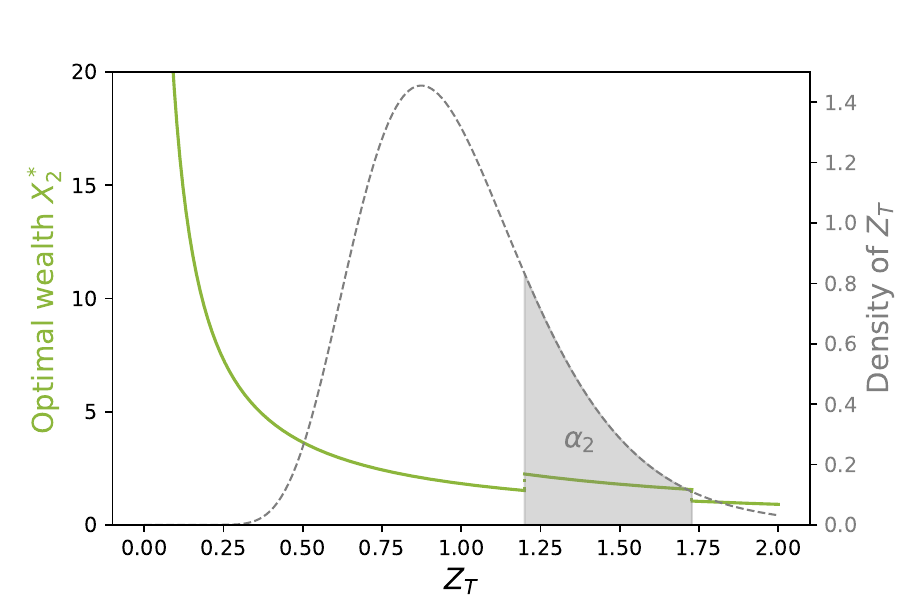}
         \caption{$A_2 = [1.2, 1.7274]$}
         \label{fig:five over x2}
     \end{subfigure}
        \caption{Terminal wealth $X_2^*$ of agent 2 in the Nash equilibrium from Theorem \ref{thm:nash_log2} in terms of $Z_T$ for $\beta_1 = 0.9$, $\alpha_2=0.2,$ and different choices of the set $A_2$.}
        \label{fig:diff_sets}
    \end{minipage}
\end{figure}

\textcolor{black}{Finally, let us take a look at the replicating strategies determined in Theorem~\ref{thm:repl_strat}. Figure~\ref{fig:strategy_log} shows a comparison of one path (i.e. for one realization of the path $(Z_t),\, t\in [0,T],$ of the state price density) of the replicating strategies in terms of invested amounts in the benchmark case without the VaR-based constraint  and for the Nash equilibrium. We displayed three different choices of the set $A$ from Theorem~\ref{thm:repl_strat}. Each set is of the form $A=\{c_1 < Z_T < c_2\}$, where $c_1$ takes values in $\{0,0.75, 1.2\}$ and $c_2$ is defined as in \eqref{eq:upper_bound}. We notice that the highest fluctuation of the amount invested into the stock appears towards the end of the time interval. This is a result from the VaR-based constraint which compares the wealth of agent 2 to the wealth of her competitor at the terminal time $T$. Thus, at the beginning of the investment period, the dominating investment motive is the own utility. Moreover, we notice that the investment behavior associated to $c_1 = 0$ and $c_1 = 0.75$ is a lot riskier in the depicted scenario than the benchmark portfolio towards the end of the time interval, while $c_1 = 1.2$ appears to be less risky as the amount invested into the stock is always smaller than the benchmark portfolio.  However, the behavior heavily depends on the specific path evolution. See also Figure~\ref{fig:strategy_realizations} which displays five different realizations of the replicating portfolio processes for $X_2^*$ from Theorem~\ref{thm:nash_log2}~c). For each path, we used $c_1 = 0$, i.e. the insured set is $A = \{Z_T \leq z_{\alpha_2}\}$. The plot shows a strong difference between the different realizations, especially when comparing the first, second, and fifth path. In general however, one can conclude that the probability constraint becomes important towards the end of the trading horizon and leads at least in some cases to risky and very volatile investment strategies. These are scenarios where towards the end we have a close competition between agents.
Figure~\ref{fig:wealth_log} displays the wealth processes associated to the strategies from Figure~\ref{fig:strategy_log}. The processes behave quite similar for the first part of the time interval. The largest difference can again be observed towards the terminal time $T$. 
Finally, Figure~\ref{fig:wealth_realizations} shows the wealth processes to the strategies shown in Figure~\ref{fig:strategy_realizations}. The target terminal wealth $\beta_1 X_1^*$ is shown as a square in the respective color at the terminal time. We notice that the target is achieved in one of the five realizations which aligns with the parameter choice of $\alpha_2 = 0.2$ in this case. Moreover, it is interesting to see that the simulation confirms that the insured set consists of scenarios where we have a good market performance and the target is hit precisely (of course this is only possible here since we assume no market frictions). In those scenarios where the target is not met, we can see that agents are quite far away from the target and underperform significantly.}

\begin{figure}
    \centering
    \includegraphics[width=\linewidth]{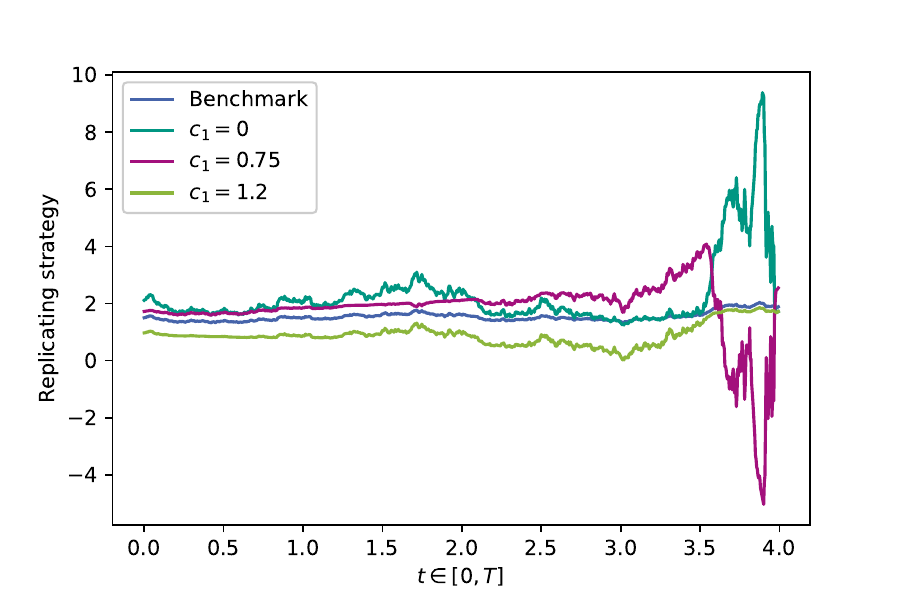}
    \caption{\textcolor{black}{Replicating strategy $\pi_2^*(t),$ $t\in [0,T]$, for the terminal wealth $X_2^*$ in Theorem~\ref{thm:nash_log2} c) for $c_1\in \{0, 0.75, 1.2\}$ and the model parameters $\mu = 0.03, \,\sigma=0.2,\, T = 4,\, \alpha_2 = 0.2,\, \beta_1 = 0.9,\, x_0^1 = 3,\, x_0^2 = 2$.}}
    \label{fig:strategy_log}
\end{figure}

\begin{figure}
    \centering
    \includegraphics[width=\linewidth]{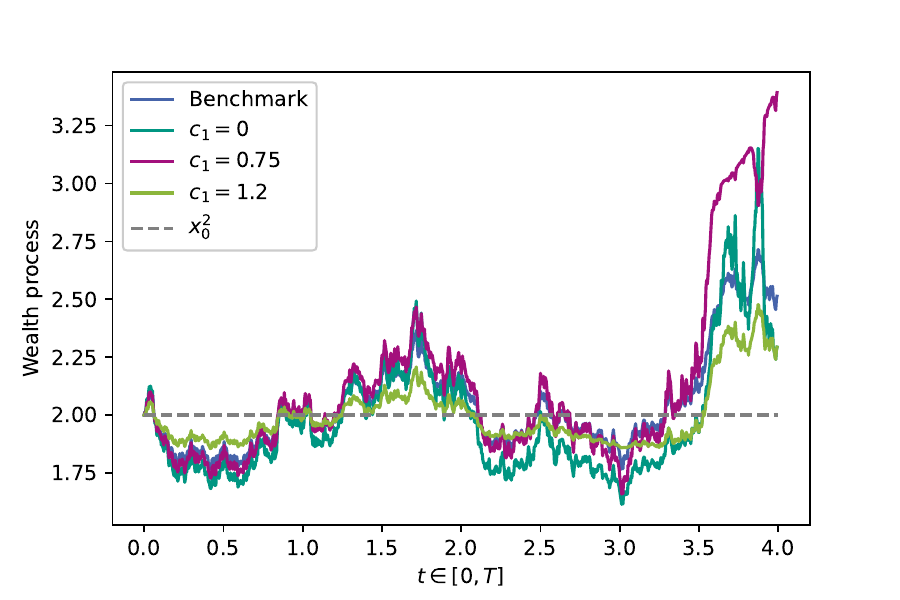}
    \caption{\textcolor{black}{Wealth process $X_2^*(t),\, t\in [0,T],$ for the terminal wealth $X_2^*$ in Theorem~\ref{thm:nash_log2} c) for $c_1\in \{0, 0.75, 1.2\}$ and the model parameters $\mu = 0.03,\, \sigma=0.2,\, T = 4,\, \alpha_2 = 0.2,\, \beta_1 = 0.9,\, x_0^1 = 3,\, x_0^2 = 2$.}}
    \label{fig:wealth_log}
\end{figure}

\begin{figure}
    \centering
    \includegraphics[width=\linewidth]{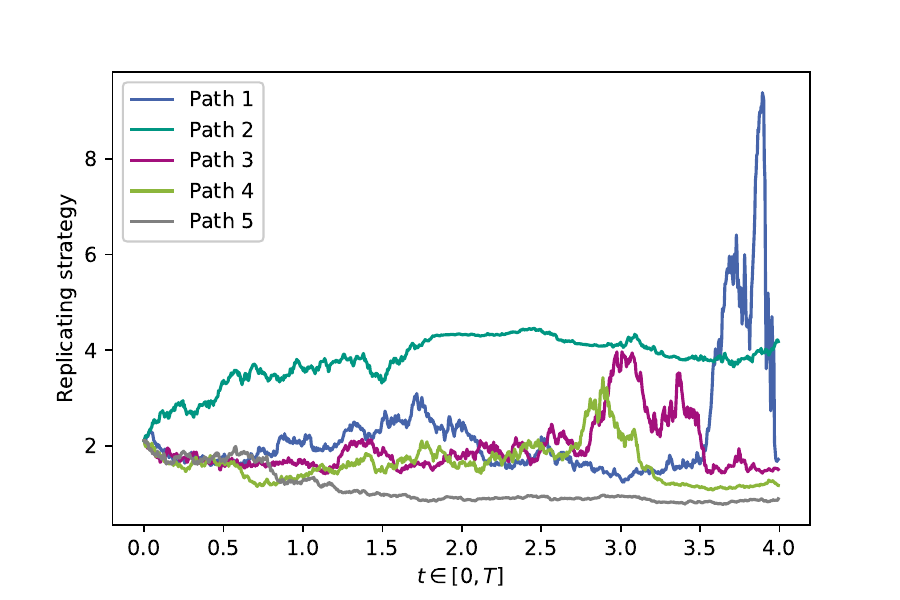}
    \caption{\textcolor{black}{Five realizations of the replicating strategy $\pi_2^*(t),$ $t\in [0,T]$, for the terminal wealth $X_2^*$ in Theorem~\ref{thm:nash_log2} c) for $c_1 = 0$ and the model parameters $\mu = 0.03,\, \sigma=0.2,\, T = 4, \,\alpha_2 = 0.2,\, \beta_1 = 0.9,\, x_0^1 = 3,\, x_0^2 = 2$.}}
    \label{fig:strategy_realizations}
\end{figure}

\begin{figure}
    \centering
    \includegraphics[width=\linewidth]{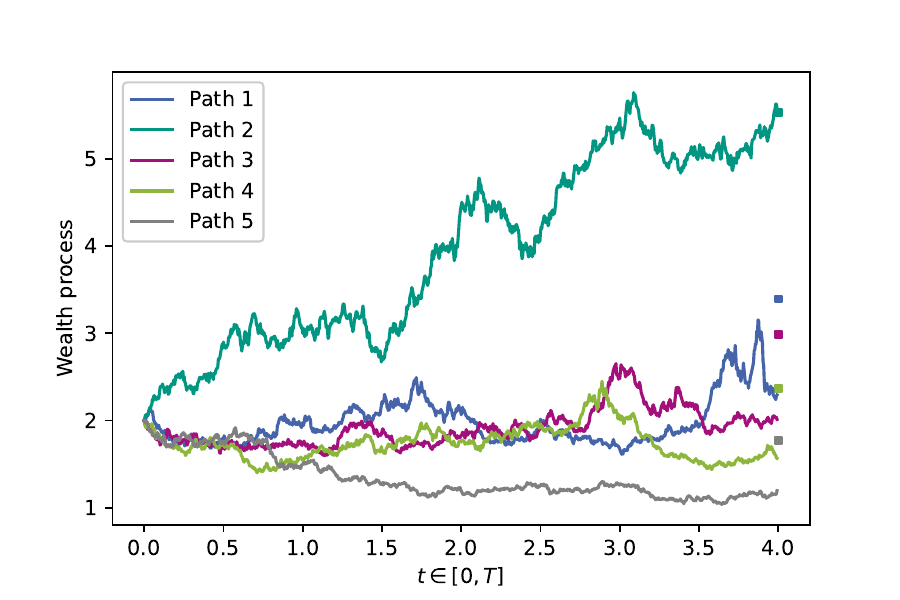}
    \caption{\textcolor{black}{Five realizations of the wealth process $X_2^*(t),$ $t\in [0,T]$, associated to the terminal wealth $X_2^*$ in Theorem~\ref{thm:nash_log2} c) for $c_1 = 0$ and the model parameters $\mu = 0.03, \,\sigma=0.2,\, T = 4,\, \alpha_2 = 0.2, \,\beta_1 = 0.9,\, x_0^1 = 3,\, x_0^2 = 2$. The weighted terminal wealth $\beta_1 X_1^*$ of agent 1 is marked by a square in the respective color.}}
    \label{fig:wealth_realizations}
\end{figure}

\subsection{\textcolor{black}{Two agent case with power utility}}

Next, we consider the Nash equilibrium from Theorem~\ref{lem:power} in the power case. Since we assumed $Z_T$ to be lognormally distributed with parameters $\nu$ and $\tau^2$, we can explicitly calculate $\varepsilon_\gamma$ and $\E\big[Z_T^{1-1/\gamma} \1 \{ Z_T \geq z_{1-\alpha_2}\}\big]$ to obtain
\begin{align*}
    \varepsilon_\gamma &= \exp\bigg( \Big(1-\frac{1}{\gamma}\Big)\nu + \frac{\tau^2}{2}\Big(1-\frac{1}{\gamma} \Big)^2 \bigg),\\
    \E\big[Z_T^{1-1/\gamma} \1 \{ Z_T \geq z_{1-\alpha_2}\}\big] &= \varepsilon_\gamma \Phi\bigg(- \frac{\mathrm{ln}(z_{1-\alpha_2}) - \nu - \tau^2(1-1/\gamma)}{\tau} \bigg),
\end{align*}
where $\Phi$ denotes the cumulative distribution function of the standard normal distribution. Figures \ref{fig:x1x2_power}, \ref{fig:nash_power_influence_alpha} and \ref{fig:nash_power_influence_beta} illustrate the Nash equilibrium in terms of $Z_T$ for the parameter choice $x_0^1 = 3$, $x_0^2 = 2$, $\gamma = 0.7$ and $\alpha_i,\beta_i$, $i=1,2,$ so that \eqref{eq:case2} holds.

Figure~\ref{fig:x1x2_power} shows the Nash equilibrium $(X_1^*, X_2^*)$ as a function of $Z_T$ for $\alpha_2 = 0.5$ and $\beta_1 = 0.9$. The purple and black solid line show the terminal wealth of agents 1 and 2 in the Nash equilibrium. For comparison, we also illustrated the unconstrained optimal terminal wealth of agent 2 as the green dashed line. Similar to the logarithmic case, the terminal wealth of agent 1 is again the solution to the standard problem without the VaR-based constraint and is thus continuous. The terminal wealth $X_2^*$ of agent 2 is smaller than the unconstrained terminal wealth for $Z_T < z_{1-\alpha_2}$ and larger for $Z_T \geq z_{1-\alpha_2}$ due to the structure of $X_2^*$ shown in Theorem~\ref{lem:power}. {\color{black} Note that due to $\gamma=0.7$ we are here in the setting where insured states are those with poor market performance.}
Figure~\ref{fig:nash_power_influence_alpha} illustrates the influence of the parameter $\alpha_2$ on $X_2^*$.  As expected, the location of the discontinuity of $X_2^*$ is decreasing in terms of $\alpha_2$ as it is simply given as the $(1-\alpha_2)$-quantile of $Z_T$. Moreover, we notice that for $Z_T < z_{\alpha_2}$ the value of $X_2^*$ is decreasing in $\alpha_2$. This results from the budget constraint $\E[Z_T X_2^*] = x_0^2$ since for larger $\alpha_2$, the value of $X_2^*$ is kept on the larger value $\beta_1 X_1^*$ on a larger interval. {\color{black} This leads to the effect that a high probability for the constraint in return implies a higher underperformance on uninsured sates. }

Finally, Figure \ref{fig:nash_power_influence_beta} shows the terminal wealth $X_2^*$ of agent 2 in the Nash equilibrium for different choices of the parameter $\beta_1$. We notice a change of order of the value of $X_2^*$ for the different choices of $\beta_1$ located at the discontinuity $z_{1-\alpha_2}.$ For $Z_T \geq z_{1-\alpha_2}$, the value of $X_2^*$ is largest for the largest choice of $\beta_1$ as the value of $X_2^*$ is kept at $\beta_1 X_1^*$ in this case. For $Z_T < z_{1-\alpha_2}$, the order of the values is opposite, i.e. the largest choice of $\beta_1$ yields the smallest value of $X_2^*$. This is again due to the budget constraint $\E[Z_T X_2^*] = x_0^2$.
 {\color{black} As in the logarithmic case  the parameter $\beta_1$ has no influence on the set of insured states.}

\begin{figure}
   \centering
    \includegraphics[width=\linewidth]{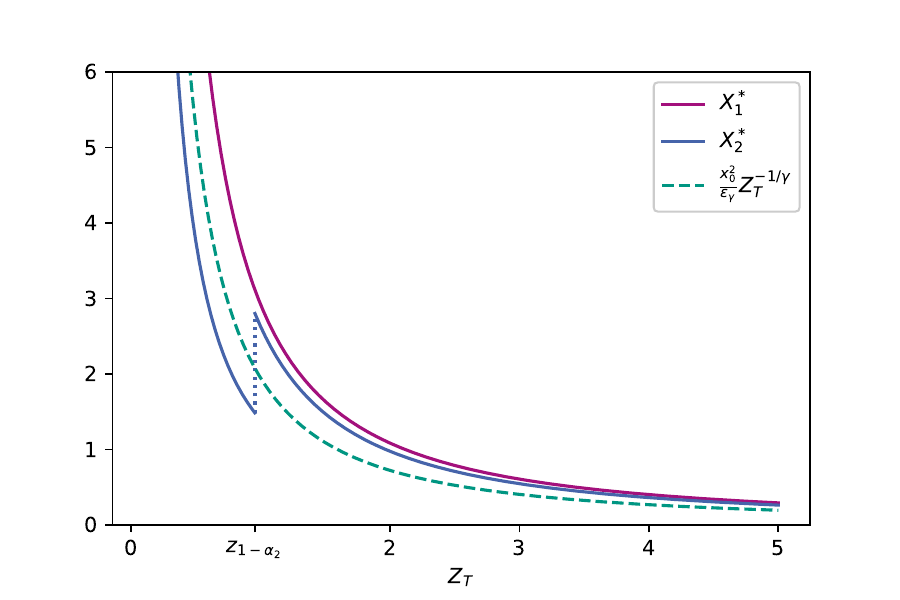}
    \caption{Nash equilibrium $(X_1^*, X_2^*)$ from Theorem \ref{lem:power} in terms of $Z_T$ for $\alpha_2 = 0.5$, $\beta_1 = 0.9$ and $\gamma = 0.7$.}
    \label{fig:x1x2_power} 
\end{figure}

\begin{figure}
    \begin{subfigure}[b]{0.49\textwidth}
         \centering
         \includegraphics[width=\textwidth]{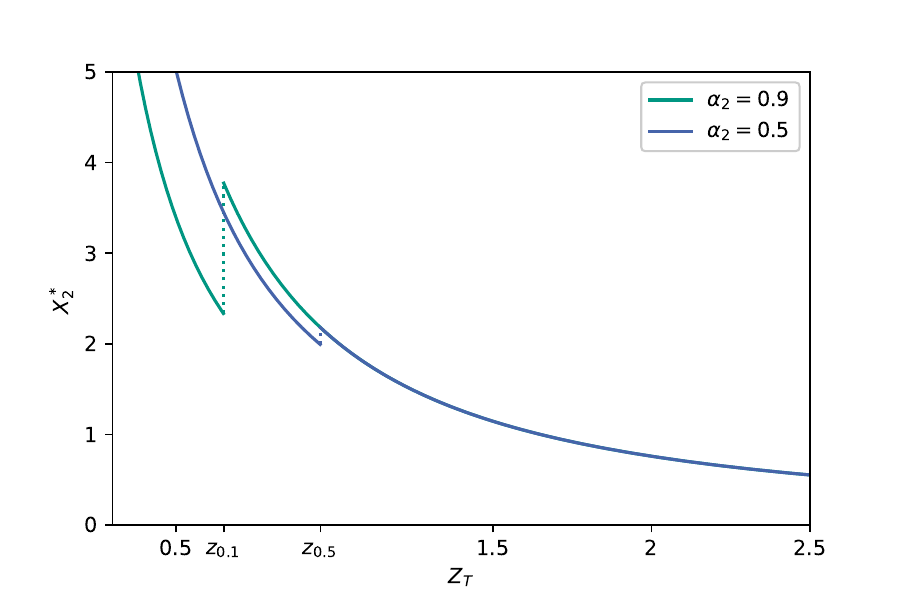}
         \caption{$\beta_1 = 0.9$ and $\alpha_2 \in \{0.5,0.9\}$}
    \label{fig:nash_power_influence_alpha}
     \end{subfigure}
     \hfill
     \begin{subfigure}[b]{0.49\textwidth}
         \centering
         \includegraphics[width=\textwidth]{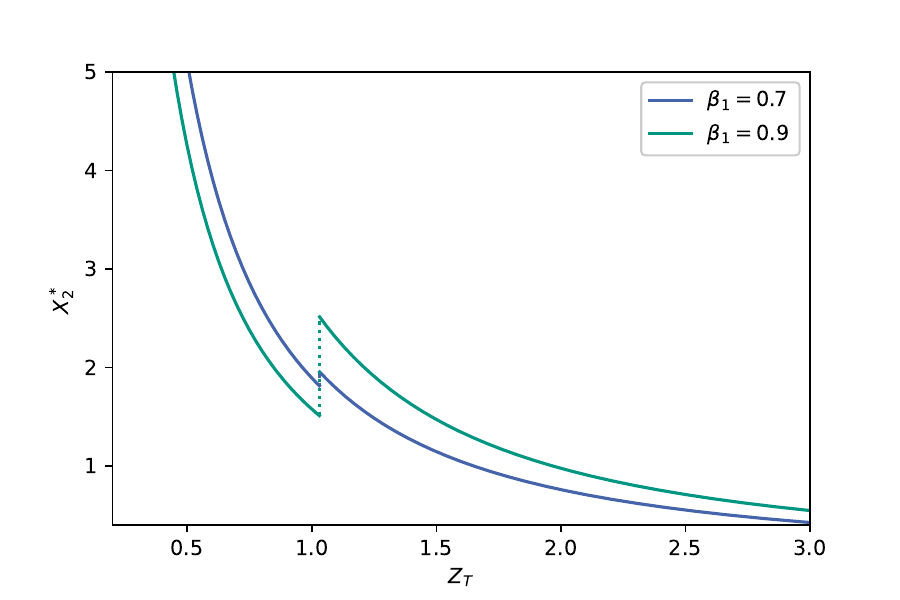}
         \caption{$\alpha_2 = 0.2$ and $\beta_1 \in \{0.7, 0.9\}$}
        \label{fig:nash_power_influence_beta}
     \end{subfigure}
     \caption{Terminal wealth $X_2^*$ of agent 2 in the Nash equilibrium from Theorem~\ref{lem:power} in terms of $Z_T$ for $\gamma = 0.7$ and different values of $\alpha_2$ and $\beta_1$.}
     \label{fig:x1x2_power_infl}
\end{figure}

\subsection{\textcolor{black}{More than two agents with logarithmic utility}}

It remains to discuss Nash equilibria for more than two agents using logarithmic utilities. Here, we restrict to the case that the sum $\alpha_1 + \ldots + \alpha_n$ is less or equal to $1$. Thus,  we can use Theorem~\ref{thm:nash_logn} to compute the terminal wealth of the agents in the Nash equilibrium. Figure \ref{fig:nash_log4} displays the terminal wealths of $n=4$ agents in the Nash equilibrium for the parameter choice $\alpha_i = 0.2$, $\beta_{{\color{black}ij}} = 0.3$ for all $i,j\in \{1,\ldots,4\}, i\neq j$ and $x_0^1 = 5$, $x_0^2 = 4$, $x_0^3 = 3$, $x_0^4 = 2$. The sets $A_i$, $i=1,\ldots,4,$ are chosen as $A_i = (z_{\alpha_1 + \ldots + \alpha_{i-1}}, z_{\alpha_1 + \ldots + \alpha_i}]$, $i=1,\ldots,4$, where $z_{\alpha_1 + \ldots + \alpha_i} := 0$ for $i=0$. We notice that the terminal wealth of agent 4 is larger than the terminal wealth of agent 3 on the set $A_4$ although agent 4 starts with a smaller initial capital. However, this results in a terminal wealth on $A_4^c$ that is significantly smaller than the optimal terminal wealth in the respective unconstrained problems. A comparison of $X_3^*$ and $X_4^*$ to the solutions of the respective unconstrained problems can be found in Figures \ref{fig:unconstrained3} and \ref{fig:unconstrained4}.
{\color{black} To sum up, we obviously see that wealthy investors are not or only marginally affected by the competition, whereas poorer agents deviate dramatically from their benchmark optimal strategy, leading to very volatile and risky investments at least in some scenarios. Moreover, in order to obtain a Nash equilibrium, the agents have to communicate and identify their sets where they want to ensure the target. }
\begin{figure}
    \centering
    \includegraphics[width=\linewidth]{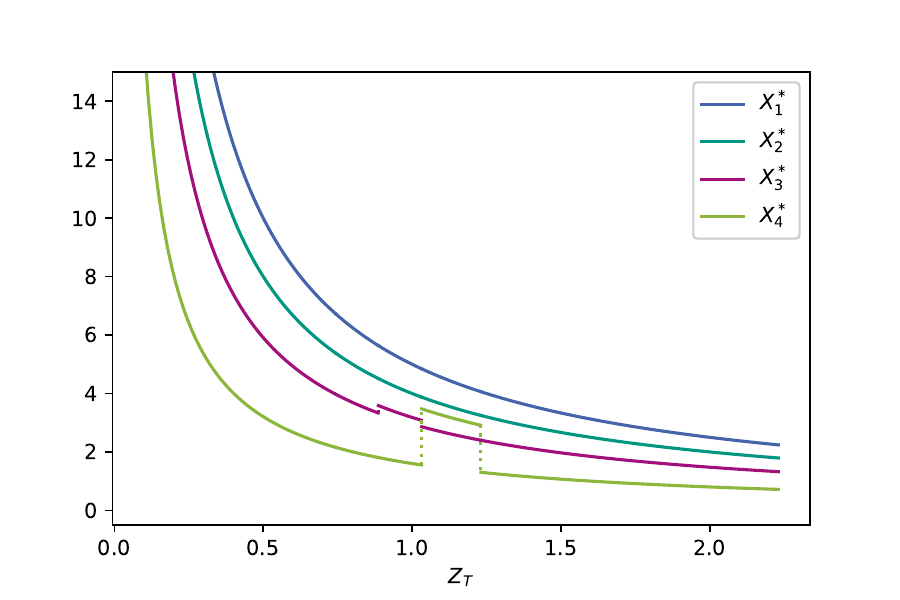}
    \caption{Nash equilibrium $(X_1^*, X_2^*, X_3^*, X_4^*)$ from Theorem \ref{thm:nash_logn} in terms of $Z_T$ for $\alpha_i = 0.2$, $\beta_{{\color{black}ij}} = 0.3$ for $i\in \{1,\ldots,4\}, i\neq j$, $x_0^1 = 5$, $x_0^2 = 4$, $x_0^3 = 3$, $x_0^4 = 2$ and $A_i = (z_{\alpha_1 + \ldots + \alpha_{i-1}}, z_{\alpha_1 + \ldots + \alpha_i}]$, $i=1,\ldots,4$, where $z_{\alpha_1 + \ldots + \alpha_i} := 0$ for $i=0$.}
    \label{fig:nash_log4}
\end{figure}

\begin{figure}
     \centering
     \begin{subfigure}[b]{0.49\textwidth}
         \centering
         \includegraphics[width=\textwidth]{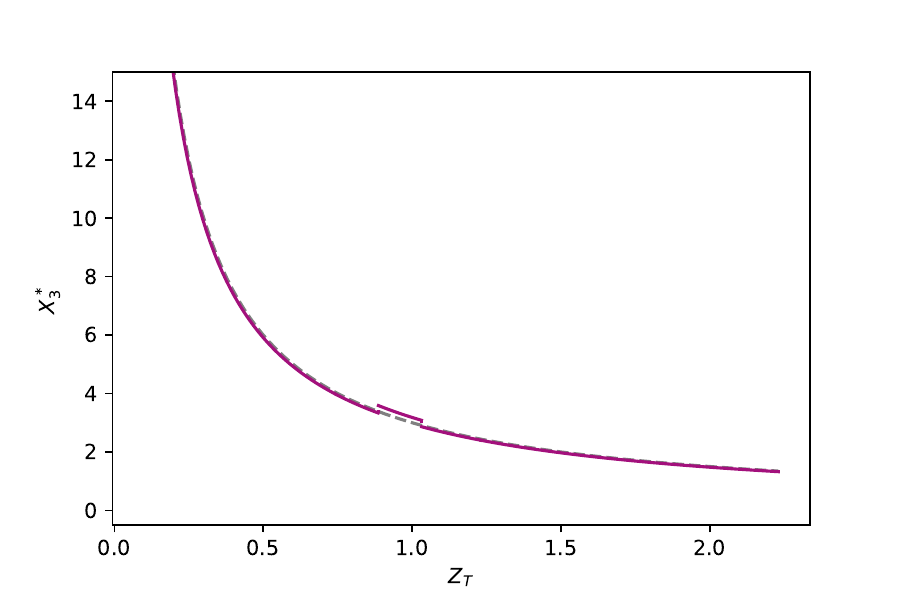}
         \caption{Agent 3}
    \label{fig:unconstrained3}
     \end{subfigure}
     \hfill
     \begin{subfigure}[b]{0.49\textwidth}
         \centering
         \includegraphics[width=\textwidth]{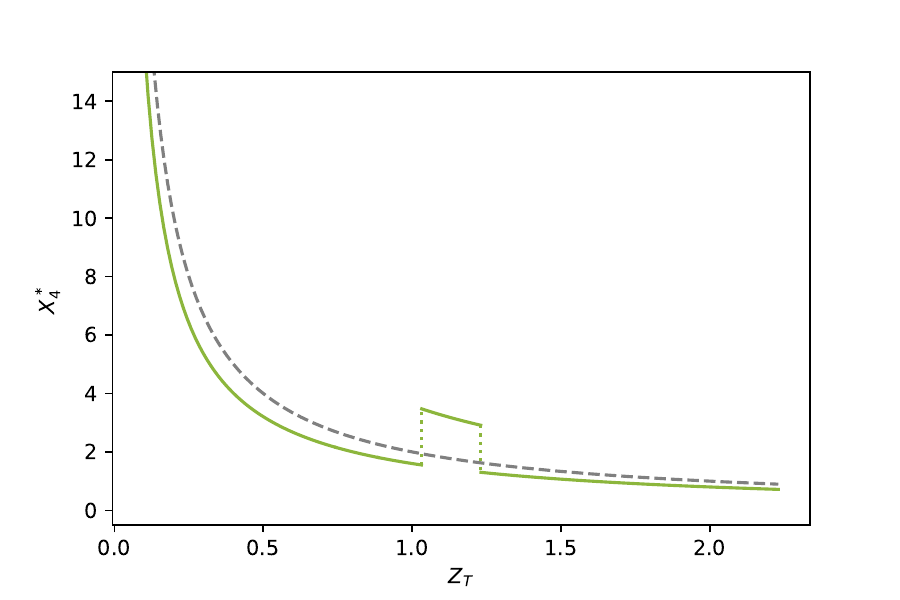}
         \caption{Agent 4}
        \label{fig:unconstrained4}
     \end{subfigure}
     \caption{Terminal wealth $X_j^*$ of agents $j$, $j\in \{3,4\},$ in the Nash equilibrium from Theorem \ref{thm:nash_logn} (solid) and the associated optimal terminal wealth in the unconstrained problem (dashed) in terms of $Z_T$.}
\end{figure}

\afterpage{\clearpage}

\newpage

\section{Conclusion}
It is meaningful to consider competing agents by introducing probability constraints instead of relative wealth targets, because this is in line with observations from psychologists and resembles incentives for fund managers. However, it turns out that finding Nash equilibria is in general more challenging. An interesting case which is analytically solvable is the case that all investors have a logarithmic utility function. In this case, Nash equilibria (if they exist) can be found for an arbitrary number of agents. In order to succeed, a number of parameter cases have to be distinguished leading either to no Nash equilibrium, a unique Nash equilibrium or an infinite number of Nash equilibria. We illustrate our findings by numerical examples. {\color{black} It turns out that wealthy investors are not or only marginally affected by the competition, whereas poorer agents deviate dramatically from their benchmark optimal strategy, leading to very volatile and risky investments at least in some scenarios.  }  For other utility functions, and in particular, if the agents have different utility functions or also different information, the problem is much more demanding and we leave it for future research.

\section{Appendix}
\subsection{Utility maximization with lower bounds I}\label{sec:A1}

Suppose  $Y\ge 0$ is an arbitrary $\mathcal{F}-$measurable random variable, $U:[0,\infty)\to \R$ a strictly increasing, strictly concave and differentiable utility. We want to solve the following problem
$$(P_Y) \left\{ \begin{array}{l}
\EE  U(X) \to \max\\
X \ge Y,\; X \mbox{ is } \mathcal{F}-\mbox{measurable,}\\
\EE[Z_TX]=x_0.
\end{array}\right.$$
Obviously $x_0\ge \EE[Z_TY]$ since otherwise it is not possible to fulfill the constraint.
We claim that the optimal solution is given as follows (cp. \cite{el2005optimal}, Prop. 2.2 for a special case).
\begin{lemma}\label{lem:bound}
The optimal solution of problem $(P_Y)$ is given by
 $$ X^* = \max\Big\{Y, I(\lambda Z_T)\Big\},$$
where $\lambda>0  $ is such that $\EE[Z_TX^*]=x_0$ and $I$ is the inverse function of $U'.$  It is the unique solution up to sets of measure zero.  
\end{lemma}
\begin{proof}
    Obviously $X^*$ is admissible for $(P_Y)$.
In order to show optimality, let $X$ be another feasible random variable. Since $U$ is concave and differentiable it holds that
$$ U(X)\le U(X^*)+U'(X^*) (X-X^*).$$
Now we have to show $\EE [U'(X^*) (X-X^*)]\le 0.$ Let $A:= \{X^*= I(\lambda Z_T)\}.$ Observe first that
$$U'(X^*) = \1_A U'( I(\lambda Z_T)) + \1_{A^c} U'(Y)= \1_A \lambda Z_T + \1_{A^c} U'(Y).$$
Thus, we obtain
\begin{align*}
U'(X^*) (X-X^*) &= \left[  \1_A \lambda Z_T + \1_{A^c} U'(Y) \right] (X-X^*)\\
&=  \lambda {Z_T}  (X-X^*) +  \1_{A^c}\big(U'(Y)-\lambda Z_T \big)(X-X^*)\\
&=  \lambda {Z_T}  (X-X^*) +  \1_{A^c}\big(U'(Y)-\lambda Z_T \big)(X-Y).
\end{align*}
The last equation is true since on $A^c$ we have $X^*=Y.$ Now we   take the expectation on both sides. Since $X$ and $X^*$ are both feasible we have $\EE[Z_T (X-X^*) ]=x_0-x_0=0.$ Thus
$$\EE [U'(X^*) (X-X^*)] =  \EE \left[ \1_{A^c}\big(U'(Y)-\lambda Z_T \big)(X-Y) \right] \le 0.  $$
The inequality is true since $X\ge Y$ and $$ U'(Y) \le \lambda Z_T \Leftrightarrow Y \ge I(\lambda Z_T)$$
which is true on $A^c.$ Finally we show uniqueness. Suppose that both $X^*$ and $\tilde X$ are optimal, i.e.\ $\EE  U(X^*)=\EE  U(\tilde X)$ and both are admissible. Consider $X := \alpha X^* + (1-\alpha) \tilde X$ for $\alpha\in (0,1).$ Obviously $X$ is again admissible and $$ \EE  U(X) > \alpha \EE  U(X^*)+(1-\alpha) \EE  U(\tilde X)$$
as long as $X^*$ and $\tilde X$ do not coincide outside a set of measure zero, which leads to the contradiction that $X$ attains a higher value.
\end{proof}

If $U(x)=\ln(x)$ we obtain $I(x)=1/x.$ Thus, with a slight misuse of the parameter $\lambda$ (instead of $\lambda$ we rather consider $1/\lambda$), we obtain in the situation of Lemma \ref{lem:bound} that $X^* = \max\Big\{Y, \lambda/ Z_T\Big\}$ where $\lambda  $ is such that $\EE[Z_TX^*]=x_0$. 

\subsection{Utility maximization with lower bounds II}\label{sec:A1+}
Now we consider the  problem with logarithmic utility where the constraint only has to be satisfied on a subset of $\Omega.$ We assume here that $Y>0.$ Then the optimization problem reads as
$$(P_A) \left\{ \begin{array}{l}
\EE  \ln (X) \to \max\\
X \ge Y \1_A,\; X \mbox{ is } \mathcal{F}-\mbox{measurable}, A \in \mathcal{F},\\
\PP(A)=\alpha,\\
\EE[Z_TX]=x_0.
\end{array}\right.$$
Note that the optimization is over $X$ and the set $A$ here.
Obviously, $x_0\ge \EE[Z_TY \1_A]$ has to be fulfilled for an $A\in \mathcal{F}$, otherwise it is not possible to fulfill the constraint. Let us define
$$ M_\lambda := \{ Y \le \lambda/Z_T\}$$
and let $\lambda_\alpha := \inf\{\lambda : \PP(M_\lambda)\ge \alpha\}.$ Note that if $\lambda_\alpha \le x_0$ then the problem is trivial and the optimal solution is given by $X^*=x_0/Z_T.$

\begin{lemma}\label{lem:bound2}
The optimal solution of problem $(P_A)$ is given by
 $$ X^* = \1_{M_{\lambda_\alpha}}\max\Big\{Y,  \lambda/Z_T\Big\}+ \1_{M_{\lambda_\alpha}^c}  \lambda/Z_T$$
where $\lambda  $ is such that $\EE[Z_TX^*]=x_0$.
\end{lemma}

\begin{proof}
    First it is slightly more convenient to transform the random variables as follows. Define
    \begin{align}
        \tilde Y := Y Z_T, \mbox{  and  } \tilde X := X Z_T.
    \end{align}
    Then, instead of $(P_A)$ we can consider
   $$(\tilde P_{ A}) \left\{ \begin{array}{l}
\EE  \ln (\tilde X) \to \max\\
\tilde X \ge \tilde Y \1_A,\; \tilde X \mbox{ is } \mathcal{F}-\mbox{measurable}, A \in \mathcal{F},\\
\PP(A)=\alpha,\\
\EE[\tilde X]=x_0.
\end{array}\right.$$

We begin with the special case that $\tilde Y$ is discrete and has finitely many different values, i.e.
$$ \tilde Y = \sum_{k=1}^m y_k 1 _{A_k}$$
for a partition $(A_k)_{k=1,\ldots,m}$ of $\Omega.$ Suppose $\tilde X$ is admissible for $(\tilde P_{A})$ and define $\tilde A := \{ \tilde X \ge \tilde Y\}$. Since $\tilde X$ is admissible we must have $\PP(\tilde A) \ge \alpha.$ Further let
\begin{align}
     B_k := A_k \cap \tilde A, \quad C_k := A_k \cap \tilde A^c.
\end{align}
I.e.\ we split the sets $A_k$ where $\tilde Y$ is constant into those parts where $\tilde{X}\geq \tilde{Y}$ and those where $\tilde{X} < \tilde{Y}$. We define now a new random variable $\hat X$ by
\begin{align}
    \hat X := \sum_{k=1}^m x_k^1 1 _{B_k}+ \sum_{k=1}^m x_k^2 1 _{C_k},
\end{align}
where 
\begin{align}
    x_k^1 := \frac{1}{\PP(B_k)} \int_{B_k} \tilde X d\PP, \quad  x_k^2 := \frac{1}{\PP(C_k)} \int_{C_k} \tilde X d\PP.
\end{align}
This means we replace $\tilde X$ on the sets $B_k, C_k$ by the corresponding expectation. Note that $\hat X$ is again admissible for $(\tilde P_{ A})$ since on $B_k$ we have $\tilde X \ge y_k$ and thus $x_k^1\ge y_k.$ Moreover, we obtain that $\EE \ln(\hat X) \ge \EE \ln(\tilde X)$ because due to the Jensen inequality we have (we denote by $\PP_{B_k}$ the conditional probability of $\PP$ given $B_k$, i.e. $\PP_{B_k}(D) = \PP(D\cap B_k)/\PP(B_k)$) that
\begin{align}
     \int_{B_k} \ln(\tilde X)  d\PP & =  \PP(B_k) \int \ln(\tilde X) d\PP_{B_k} \le  \PP(B_k) \ln\Big( \int \tilde Xd\PP_{B_k}\Big) =  \PP(B_k) \ln(x_k^1)
\end{align}
and the same for the sets $C_k$. Summing up these integrals we see that the expected utility of $\hat X$ is not less than the expected utility for $\tilde X.$ Thus, we can restrict the optimization to random variables $\hat X$ which are discrete and have finitely many positive values.

Fix an admissible discrete $\hat X$ and assume that there exists a measurable set $F\subset B_k$ for an arbitrary $k$ and a measurable set $\tilde F\subset C_j$ for $j\neq k$ and with $\PP(F)=\PP(\tilde F)>0$ such that $y_k>y_j.$ Note that $\PP(F)$ may be arbitrary small. This means that the constraint is satisfied on a larger level $y_k$ whereas it is not satisfied on a smaller level $y_j.$ By construction, the random variable $\hat X$ takes value $x_k^1$ on set $B_k$ and value $x_j^2$ on $C_j.$ Thus we have 
$$ x_j^2 < y_j < y_k \le x_k^1. $$
Now define the random variable
$$ X^* := \hat X \1_{(F\cup \tilde F)^c} + y_j \1_{\tilde F}+ (x_j^2+x_k^1-y_j) \1_F.$$
Note that $x_j^2+x_k^1-y_j>0$ and $\EE \hat X=\EE X^*.$ Moreover, $X^*$ also satisfies the constraint due to our construction. Let us consider the difference in expected utility of $\hat X$ and $X^*:$
\begin{align}
    \EE \ln(X^*) - \EE \ln(\hat X)= \PP(F) \Big(\ln(y_j)+\ln(x_j^2+x_k^1-y_j)-\ln(x_k^1)-\ln(x_j^2)\Big)>0.
\end{align}
The latter inequality follows since $\R^m\ni (x_1,\ldots,x_m)\mapsto \sum_{k=1}^m \ln(x_k)$ is Schur-concave (see \cite{marshall1979inequalities}, Chapt.1, Sec. A). Thus, we obtain that it is always better to satisfy the constraint on a set where $Y$ takes the smallest values. This implies the statement for discrete $Y$. In order to show the statement for arbitrary $Y$, approximate $Y$ by a sequence of discrete $(Y_n)$ almost surely. Taking the limit $n\to \infty$ then implies the general result.
\end{proof}

\subsection{Maximizing L}\label{sec:A2}
Recall that 
 $$ L(X,\lambda_2,\eta_2) := U(X) -\lambda_2 Z_T X +\eta_2 \1_{[X\ge \beta_1 I(\kappa Z_T)]}$$
 with $\lambda_2, \eta_2 \ge 0$ fixed. We show that $L$ is maximized by $X_2^*$ from \eqref{eq:NEpower2}. Obviously the maximum points can either be $I(\lambda_2 Z_T)$ or $\beta_1 I(\kappa Z_T)$ where $I(\lambda_2 Z_T) < \beta_1 I(\kappa Z_T).$ We can compare the two possible values of the function $L$:
 \begin{align*}
     F(Z_T) := U(I(\lambda_2 Z_T))- \lambda_2 Z_T I(\lambda_2 Z_T) - \Big\{U(\beta_1 I(\kappa Z_T)) -\lambda_2 Z_T \beta_1 I(\kappa Z_T)+\eta_2 \Big\}. 
 \end{align*}
Differentiating $F$ yields
\begin{align*}
    F'(Z_T) &= \lambda_2 \big(\beta_1 I(\kappa Z_T) -  I(\lambda_2 Z_T)\big) - \beta_1 \kappa Z_T I'(\kappa Z_T)(\kappa \beta_1^{-\gamma}-\lambda_2) \\
    &= Z_T^{-1/\gamma} \bigg(\lambda_2 \Big(\beta_1 \kappa^{-1/\gamma} - \lambda_2^{-1/\gamma} \Big) + \frac{\beta_1 \kappa^{-1/\gamma}}{\gamma} \Big(\kappa \beta_1^{-\gamma} - \lambda_2 \Big) \bigg) =: Z_T^{-1/\gamma} f(\lambda_2).
\end{align*}

Since $Z_T>0$, it suffices to consider the sign of $f(\lambda_2)$ to discuss the monotonicity of $F$. We observe that 
$$ f\big(\kappa \beta_1^{-\gamma} \big) = 0.$$
Moreover, differentiating $f$ yields 
\begin{align*}
    f'(\lambda_2) &= \Big(1-1/\gamma \Big)\Big(\beta_1 \kappa^{-1/\gamma} - \lambda_2^{-1/\gamma}\Big),
\end{align*}
so that $f'(\lambda_2) < 0$ for all $\lambda_2 > \kappa \beta_1^{-\gamma}$ {\color{black}(in case $0<\gamma <1$. The inequality reverses if $\gamma >1$)}. Thus, since we already saw in the proof of Theorem~\ref{lem:power} that $\lambda_2 > \kappa\beta_1^{-\gamma}$, we deduce that $F'(Z_T) < 0$ for all values of $Z_T$. Moreover, by definition of $\eta_2$, we have $F(z_{1-\alpha_2}) = 0$ which implies that $X_2^*$ maximizes the function $L$.

\subsection{\textcolor{black}{Replicating the terminal wealth in the Nash equilibrium}\label{sec:A_strategy}}
\textcolor{black}{
In the following, we provide an auxiliary statement used to find the replicating strategies for the Nash equilibria found in Theorem~\ref{thm:nash_log2} in the special case of a Black-Scholes market. Thus, we use the financial market explained at the beginning of Section~\ref{sec:FM}. Additionally, we introduce the state price density process given by 
$$ Z_t = \exp\Big(-\theta^\top W_t - \frac{1}{2}\Vert \theta \Vert^2 t \Big),\quad 0\leq t \leq T.$$
Then a straightforward application of Theorem~E.1 in \cite{jin2008behavioral} yields the following result, where $\varphi$ and $\Phi$ denote the density and cumulative distribution function of the standard normal distribution. }

\textcolor{black}{\begin{lemma}\label{lem:replicating}
Let $0\leq c_1 < c_2 \leq \infty$ and $X = Z_T^{-1} \mathbbm{1}\{ c_1 < Z_T < c_2\}$. If $0<c_1 < c_2 < \infty$, the wealth-portfolio pair replicating $X$ is given by 
\begin{align*}
    X(t) &= \frac{1}{Z_t}\Bigg( \Phi\Bigg(\frac{\log(c_2) - \log(Z_t) + \frac{1}{2}\Vert \theta \Vert^2 (T-t)}{\Vert \theta \Vert \sqrt{T-t}} \Bigg) - \Phi\Bigg(\frac{\log(c_1) - \log(Z_t) + \frac{1}{2}\Vert \theta \Vert^2 (T-t)}{\Vert \theta \Vert \sqrt{T-t}} \Bigg) \Bigg),\\
    \pi(t) &= \Bigg[  X(t) + \frac{1}{Z_t \Vert \theta \Vert \sqrt{T-t}}\Bigg( \varphi \Bigg( \frac{\log(c_2) - \log(Z_t) + \frac{1}{2}\Vert \theta\Vert^2 (T-t)}{\Vert \theta \Vert \sqrt{T-t}} \Bigg) \\
    &  \hspace{4.5cm} -  \varphi \Bigg( \frac{\log(c_1) - \log(Z_t) + \frac{1}{2}\Vert \theta\Vert^2 (T-t)}{\Vert \theta \Vert \sqrt{T-t}} \Bigg)\Bigg)\Bigg]\big(\sigma \sigma^\top\big)^{-1} \mu.
\end{align*}
The portfolio strategy $\pi$ describes the amount of money invested into the $d$ stocks at each time $t\in [0,T)$. Results for the cases $c_1 = 0$ and $c_2 = \infty$ can be obtained by taking the respective limits. 
\end{lemma}}
\textcolor{black}{\begin{proof}
    The result is a straightforward application of Theorem~E.1 in \cite{jin2008behavioral} using the following auxiliary calculation:
    \begin{align*}
        \int_a^b \frac{1}{y} \varphi\Big(\frac{\log(y) - \nu}{\tau} \Big) \mathrm{d}y &= \tau \int_a^b  \underbrace{\frac{1}{\tau y} \varphi\Big(\frac{\log(y) - \nu}{\tau} \Big)}_{(*)}\mathrm{d}y \\
        &= \tau \left(\Phi\Big(\frac{\log(b) - \nu}{\tau} \Big)  - \Phi\Big(\frac{\log(a) - \nu}{\tau} \Big)\right),
    \end{align*}
    where we used that $(*)$ is the density function of a lognormal distribution with parameters $\nu$ and $\tau^2$ and that the associated cumulative distribution function is given by $\Phi((\log(\cdot) - \nu)/ \tau)$. 
\end{proof} }
\bibliographystyle{apalike}
\bibliography{literatur}

\end{document}